\documentclass[12pt]{amsart}

\usepackage{amsmath,amsthm,amsfonts,amssymb,anysize,bbm,euscript,enumitem,stmaryrd,mathrsfs,xcolor,mathscinet,mathtools,combelow,tikz,tikz-cd,caption}
\newtheorem{lemma}{Lemma}[section]
\newtheorem{theorem}[lemma]{Theorem}
\newtheorem{conjecture}[lemma]{Conjecture}
\newtheorem{corollary}[lemma]{Corollary}
\newtheorem{proposition}[lemma]{Proposition}
\newtheorem{example}[lemma]{Example}
\newtheorem{remark}[lemma]{Remark}
\newtheorem{definition}[lemma]{Definition}
\newtheorem{problem}[lemma]{Problem}

\newcommand{\comment}[1]{}
\newcommand{\A}{\mathbb{A}}

\newcommand{\C}{\mathbb{C}}
\newcommand{\N}{\mathbb{N}}
\newcommand{\Z}{\mathbb{Z}}
\newcommand{\Q}{\mathbb{Q}}

\newcommand{\CA}{\mathcal{A}}

\newcommand{\CE}{\mathcal{E}}

\newcommand{\CK}{\mathcal{K}}
\newcommand{\CL}{\mathcal{L}}

\newcommand{\CU}{\mathcal{U}}
\newcommand{\Shuf}{\mathcal{S}}
\newcommand{\SBim}{\mathrm{SBim}}
\newcommand{\Hilb}{\mathrm{Hilb}}

\newcommand{\one}{\mathbbm{1}}

\newcommand{\CS}{\EuScript C}
\newcommand{\DS}{\EuScript D}

\newcommand{\Hom}{\operatorname{Hom}}

\newcommand{\Br}{\operatorname{Br}}
\newcommand{\Coh}{\operatorname{Coh}}
\newcommand{\D}{\operatorname{D}}

\newcommand{\Sk}{\operatorname{Sk}}

\newcommand{\Tr}{\operatorname{Tr}}

\newcommand{\ee}{\mathbf{e}}

\newcommand{\dd}{\mathbf{d}}

\newcommand{\cox}{\text{cox}}

\newcommand{\ABr}{\mathrm{ABr}}
\newcommand{\AH}{\mathrm{AH}}
\newcommand{\ASBim}{\mathrm{ASBim}}
 
\newcommand{\R}{\mathbb{R}} 
\newcommand{\tR}{\widetilde{R}}

\newcommand{\bstar}{\overline{\star}}

\newcommand{\aaa}{\mathbf{a}}
\newcommand{\bb}{\mathbf{b}}
\newcommand{\cc}{\mathbf{c}}

\newcommand{\loc}{\mathrm{loc}}

\newcommand{\Comm}{\mathrm{Comm}}
\newcommand{\FComm}{\mathrm{FComm}}
\newcommand{\MF}{\mathrm{MF}}

\newcommand{\EE}{\mathscr{E}}

\author{Eugene Gorsky}
\address{Department of Mathematics, University of California\\ One Shields Avenue, Davis CA 94702}
\email{egorskiy@math.ucdavis.edu}
\author{Andrei Negu\cb{t}}
\address{MIT, Department of Mathematics, Cambridge, MA, USA}
\address{Simion Stoilow Institute of Mathematics, Bucharest, Romania}
\email{anegut@mit.edu}

\title{The Trace of the affine Hecke category}

\begin{document}

\begin{abstract}

We compare the (horizontal) trace of the affine Hecke category with the elliptic Hall algebra, thus obtaining an ``affine" version of the construction of \cite{GHW}. Explicitly, we show that the aforementioned trace is generated by the objects $E_{\dd} = \text{Tr}(Y_1^{d_1} \dots Y_n^{d_n} T_1 \dots T_{n-1})$ as $\dd = (d_1,\dots,d_n) \in \mathbb{Z}^n$, where $Y_i$ denote the Wakimoto objects of \cite{Elias} and $T_i$ denote Rouquier complexes. We compute certain categorical commutators between the $E_{\dd}$'s and show that they match the categorical commutators between the sheaves $\CE_{\dd}$ on the flag commuting stack, that were considered in \cite{Hecke}. At the level of $K$-theory, these commutators yield a certain integral form $\widetilde{\CA}$ of the elliptic Hall algebra, which we can thus map to the $K$-theory of the trace of the affine Hecke category.

\end{abstract}

\maketitle

\section{Introduction}

\subsection{Motivation}

The HOMFLY-PT skein algebra $\text{Sk}(\mathbb{T})$ of the torus is generated by framed links in the thickened torus $\mathbb{T}\times [0,1]$ modulo the skein relation (which depends on a parameter $q$), and the multiplication is given by vertical stacking:
$$
\mathbb{T}\times [0,1]\sqcup \mathbb{T}\times [1,2]\to \mathbb{T}\times [0,2].
$$ Morton and Samuelson (\cite{MS}) proved  that:
\begin{equation}
\label{eqn:ms}
\text{Sk}(\mathbb{T}) \cong \CE_{q,q^{-1}}
\end{equation}
where $\CE_{q_1,q_2}$ is the {\em elliptic Hall algebra} of Burban and Schiffmann (\cite{BS}). Under the isomorphism \eqref{eqn:ms}, the $(m,n)$ torus knot goes to the generator $P_{m,n} \in \CE_{q,q^{-1}}$, for any coprime integers $m$ and $n$. The skein algebra is naturally bigraded by $H_1(\mathbb{T})\simeq \Z^2$, and $P_{m,n}$ has bidegree $(m,n)$. In this paper we will mostly focus on the {\em positive half} \ $\text{Sk}_+(\mathbb{T})$ of the skein algebra generated by $P_{m,n}$ with $n>0$. 

\bigskip

\textbf{The main purpose of this paper is to propose a way to categorify \eqref{eqn:ms}}.

\bigskip

For the left-hand side of \eqref{eqn:ms}, we recall that the {\em affine braid group} $\ABr_n$ can be defined as the mapping class group of the annulus $\mathbb{A}$ with $n$ marked points. Thus, an affine braid naturally lives inside $\mathbb{A}\times [0,1]$, and we can consider its {\em annular closure} in $\mathbb{A}\times S^1\simeq \mathbb{T}\times [0,1]$. In other words, conjugacy classes of affine braids naturally correspond to links in the thickened torus. Furthermore, the skein multiplication corresponds to placing one annulus inside another and considering the homomorphism $\ABr_n\times \ABr_k\to \ABr_{n+k}$. The affine braid group  $\ABr_n$ is naturally graded by the rotation number $m$ around the annulus, and corresponds to the component of $\Sk_{+}(\mathbb{T})$ of bidegree $(m,n)$.  

The affine braid group $\ABr_n$ is categorified by $\ASBim_n$, the category of affine Soergel bimodules \cite{Elias,MT}. The operation of annular closure corresponds to the notion of (derived) {\em horizontal trace} $\Tr$ developed in 
\cite{BHLZ,BPW,GHW,GW}. Therefore, based on the discussion above, we propose to categorify the skein of the torus by the category $\Tr(\ASBim_n)$: 
\begin{center}
\begin{tikzcd}
\ASBim_n \arrow{r} \arrow{d} & \Tr(\ASBim_n) \arrow[dotted]{d}\\
\ABr_n \arrow{r} & \Sk_{+}(\mathbb{T})
\end{tikzcd}
\end{center}
(the diagram above is for motivational purposes only: the vertical arrows denote ``categorification" and the horizontal arrows denote ``closing annular braids").

For the right-hand side of \eqref{eqn:ms}, we recall that Schiffmann and Vasserot (\cite{SV}) explained that (a certain integral form of) $\CE_{q_1,q_2}$ can be naturally mapped to:
$$
K = \bigoplus_{n=0}^\infty K_{\C^* \times \C^*}(\text{Comm}_n)
$$
where the commuting stack $\text{Comm}_n$ is defined in \eqref{eqn:commuting stack}. It is natural to categorify $K$ by the derived category of equivariant coherent sheaves on the derived stack $\text{Comm}_n$. Therefore, one approach to categorifying \eqref{eqn:ms} would be to rigorously state and prove the following. 

\begin{problem}
\label{prob:main}

Construct a functor:
\begin{equation}
\label{eqn:functor}
\Tr(\ASBim_n) \longrightarrow D^b_{\C^* \times \C^*}(\Coh(\Comm_n))
\end{equation}
and identify its essential image. 

\end{problem}

In the following Subsections, we will explain some constraints and properties that one expects from the functor \eqref{eqn:functor}, and perform the construction at the level of $K$-theory (see Theorem \ref{thm:main}). An analogue of the functor \eqref{eqn:functor} was constructed by Oblomkov and Rozansky using the language of matrix factorizations (see Subsection \ref{subsec:OR} for details).

\subsection{Affine Soergel bimodules}

The extended affine braid group $\ABr_n$ of type $\widetilde{A_{n-1}}$ can be interpreted as the group of $n$-strand braids in the punctured plane or, equivalently, braids with a pole. We will denote the pole by a thick line. 
The generators of $\ABr_n$ are denoted by $\{\sigma_i\}_{i \in \{1,\dots,n-1\}}$ and $\{y_i\}_{i \in \{1,\dots,n\}}$, and they are represented in Figure \ref{fig: affine braid} below:
\begin{figure}[ht!]
\begin{tikzpicture}
\draw[line width=3] (0,0)--(0,2);
\draw (1,0)--(1,2);
\draw (2,1) node {$\cdots$};
\draw (4,0) .. controls (4,1) and (3,1).. (3,2);
\draw[line width=5,white] (3,0) .. controls (3,1) and (4,1).. (4,2);
\draw (3,0) .. controls (3,1) and (4,1).. (4,2);
\draw (5,1) node {$\cdots$};
\draw (6,0)--(6,2);
\draw (1,-0.2) node {$\scriptsize{1}$};
\draw (3,-0.2) node {$\scriptsize{i}$};
\draw (4,-0.2) node {$\scriptsize{i+1}$};
\draw (6,-0.2) node {$\scriptsize{n}$};
\end{tikzpicture}
\qquad \qquad \qquad
\begin{tikzpicture}
\draw (-0.2,1)..controls (-0.2,1.4) and (4,1.6)..(4,2);
\draw[line width=5,white] (0,0)--(0,2);
\draw[line width=3] (0,0)--(0,2);
\draw[line width=5,white] (1,0)--(1,2);
\draw (1,0)--(1,2);
\draw (2,1) node {$\cdots$};
\draw[line width=5,white] (3,0)--(3,2);
\draw (3,0)--(3,2);
 \draw [line width=5,white ](-0.2,1)..controls (-0.2,0.6) and (4,0.4)..(4,0);
\draw (-0.2,1)..controls (-0.2,0.6) and (4,0.4)..(4,0);
\draw (5,1) node {$\cdots$};
\draw (6,0)--(6,2);
\draw (1,-0.2) node {$\scriptsize{1}$};
\draw (3,-0.2) node {$\scriptsize{i-1}$};
\draw (4,-0.2) node {$\scriptsize{i}$};
\draw (6,-0.2) node {$\scriptsize{n}$};
\end{tikzpicture}
\captionsetup{justification=centering}
\caption{Generators of the affine braid group: $\sigma_i$ (left) and $y_i$ (right)}
\label{fig: affine braid}
\end{figure}
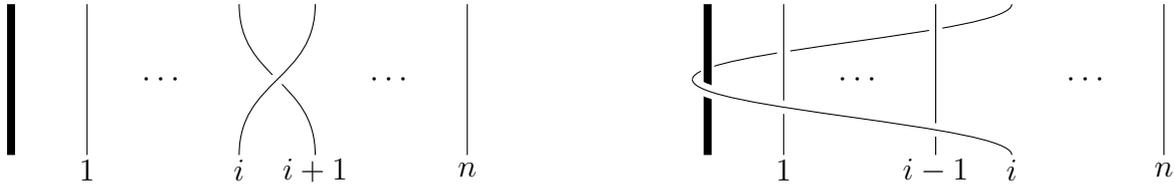

In \cite{Elias},  Elias constructed a monoidal categorification of $\ABr_n$ using the homotopy category of complexes of affine Soergel bimodules $\CK(\ASBim_n)$, see Section \ref{sec: affine sbim} for details. In particular, $\sigma_i$ correspond to standard Rouquier complexes $T_i$
(\cite{Rouquier}), while $y_i$ correspond to the so-called Wakimoto objects $Y_i$. The objects $Y_i$ commute and so for any vector of integers $\dd=(d_1,\ldots,d_n)$ one has a well-defined object $Y^{\dd} = Y_1^{d_1}\cdots Y_n^{d_n}\in \CK(\ASBim_n)$, also called a Wakimoto object in \cite{Elias}. More generally, we can tensor Wakimoto objects by the Rouquier complexes of Coxeter braids, and denote the resulting object by:
\begin{equation}
    \label{eqn:coxeter braid}
Y^{\dd} T_{\cox_n} = Y_1^{d_1} \dots Y_n^{d_n} T_1 \dots T_{n-1} \in \CK(\ASBim_n)
\end{equation}
Consider any composition $n_1+\dots+n_r=n$ and any $r$-tuple of integer vectors: 
\begin{equation}
\label{eqn:vector of integers}
\Big\{ \dd^i = (d_{n_1+\dots+n_{i-1}+1},\dots,d_{n_1+\dots+n_i}) \Big\}_{1\leq i \leq r}
\end{equation}
Then we may generalize \eqref{eqn:coxeter braid} by considering the objects:
\begin{multline}
    \label{eqn:general braid}
Y^{\dd^1} T_{\cox_{r_1}} \star \dots \star Y^{\dd^r} T_{\cox_{n_r}}  : = \\ Y_1^{d_1} \dots Y_n^{d_n} \prod_{i=1}^r T_{n_1+\dots+n_{i-1}+1} \dots T_{n_1+\dots+n_i-1} \in \CK(\ASBim_n)
\end{multline}

\begin{remark}

We expect that the objects \eqref{eqn:general braid} can be obtained from the objects \eqref{eqn:coxeter braid} by successive applications of yet-undefined bifunctors:
\begin{equation}
\label{eqn:unconstructed product before tr}
\CK(\ASBim_n) \otimes \CK(\ASBim_k) \xrightarrow{\star} \CK(\ASBim_{n+k})
\end{equation}
whose action on objects matches the homomorphisms $\ABr_n\times \ABr_k\to \ABr_{n+k}$ described in the previous Subsection. If $\alpha\in \CK(\ASBim_n)$ and $\beta\in \CK(\ASBim_k)$ are images of affine braids, then the object $\alpha\star \beta$ is well defined in $\CK(\ASBim_{n+k})$ (this follows from the braid relations). However, defining $\star$ on morphisms remains an open problem. 

\end{remark}

\subsection{Derived trace}

We will use the formalism of derived horizontal traces developed in \cite{GHW}. Given a monoidal dg category $\CS$, one constructs another dg category $\Tr(\CS)$ with the following properties:
\begin{itemize}
\item[(a)] There is a well defined dg functor $\Tr:\CS\to \Tr(\CS)$. In particular, morphisms in $\CS$ correspond to morphisms in $\Tr(\CS)$ and their cones are sent to cones.
\item[(b)] Assuming that $\CS$ has duals, one has isomorphisms $\Tr(XY)\simeq \Tr(YX)$ which are natural in both $X$ and $Y$.
\item[(c)] The endomorphism (dg) algebra of $\Tr(\one)$ is isomorphic to the Hochschild homology of $\CS$ equipped with the shuffle product.
\end{itemize}

We refer to \cite{GHW} and Section \ref{sec: trace} for more details. Property (b) allows us to interpret the traces of objects corresponding to braids in $\CK(\ASBim_n)$ as their annular closures, as in Figure \ref{fig: closing affine braid}. As the complement of the closure of the pole in $\left(\R^2\setminus\{*\}\right)\times \R$ is homeomorphic to the thickened torus $\mathbb{T}$, annular closures such as Figure \ref{fig: closing affine braid} yield elements of $\text{Sk}(\mathbb{T})$.

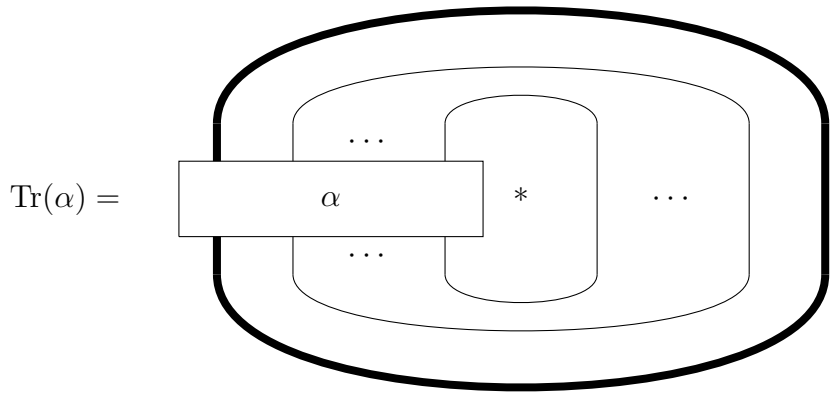
\begin{figure}[ht!]
\begin{tikzpicture}
\draw (-2,1) node {$\Tr(\alpha)=$};
\draw [line width=3] (0,0)--(0,0.5);
\draw (1,0)--(1,0.5);
\draw (2,0.25) node {$\cdots$};
\draw (3,0)--(3,0.5);
\draw (-0.5,0.5)--(-0.5,1.5)--(3.5,1.5)--(3.5,0.5)--(-0.5,0.5);
\draw (1.5,1) node {$\alpha$};
\draw [line width=3] (0,1.5)--(0,2);
\draw (1,1.5)--(1,2);
\draw (2,1.75) node {$\cdots$};
\draw (3,1.5)--(3,2);
\draw [line width=3] (0,2)..controls (0,4) and (8,4)..(8,2);
\draw (1,2)..controls (1,3) and (7,3)..(7,2);
\draw (3,2)..controls (3,2.5) and (5,2.5)..(5,2);
\draw (5,0)--(5,2);
\draw (7,0)--(7,2);
\draw  [line width=3] (8,0)--(8,2);
\draw (6,1) node {$\cdots$};
\draw [line width=3] (0,0)..controls (0,-2) and (8,-2)..(8,0);
\draw (1,0)..controls (1,-1) and (7,-1)..(7,0);
\draw (3,0)..controls (3,-0.5) and (5,-0.5)..(5,0);
\draw (4,1) node {*};
\end{tikzpicture}
\captionsetup{justification=centering}
\caption{Closing an affine braid}
\label{fig: closing affine braid}
\end{figure}

Another important feature of $\Tr(\CS)$ is its universal trace-like property. Assume that $\DS$ is another dg category and $F:\CS\to \DS$ is a trace-like dg functor (see \cite{GHW} for the precise definitions), in particular, one has isomorphisms $F(XY)\simeq F(YX)$ which are natural in $X$ and $Y$. Then $F$ factors through $\Tr(\CS)$ and fits into the commutative diagram:
\begin{center}
    \begin{tikzcd}
     \CS \arrow{dr}{F} \arrow{r}{\Tr}& \Tr(\CS)\arrow[dashed]{d}{\widetilde{F}}\\
         &   \DS
    \end{tikzcd}
\end{center}

In this paper we will be mostly interested in the trace of the dg category $\CS  = \CK(\ASBim_n)$ which we will abbreviate as $\Tr(\ASBim_n)$. The traces of the  objects \eqref{eqn:coxeter braid} and \eqref{eqn:general braid} will be denoted as follows:
\begin{equation}
\label{eqn:coxeter trace}
E_{\dd}:=\Tr(Y^{\dd}T_{\cox_n})
\end{equation}
and, in the notation of \eqref{eqn:vector of integers}:
\begin{equation}
\label{eqn:general trace}
E_{\dd^1}\star \cdots  \star E_{\dd^r}:=\Tr\left(Y^{\dd_1}T_{\cox_{n_1}} \star \cdots\star Y^{\dd^r} T_{\cox_{n_r}}\right).
\end{equation}
Similarly to \eqref{eqn:unconstructed product before tr} we expect that there exist bifunctors:
\begin{equation}
\label{eqn:unconstructed product}
\Tr(\ASBim_n) \otimes \Tr(\ASBim_k) \xrightarrow{\star} \Tr(\ASBim_{n+k}),
\end{equation}
Using the universal property of $\Tr(\CS)$, we can rephrase Problem \ref{prob:main} as follows:

\begin{problem}
\label{problem:construct}
Construct a trace-like functor:
\begin{equation}
\label{eqn:problem}
\ASBim_n \longrightarrow D^b_{\C^* \times \C^*}(\Coh(\Comm_n))
\end{equation}
and identify its essential image. 
\end{problem}

\subsection{The commuting and flag commuting stacks}

As we will recall in Section \ref{sec:eha}, the commuting stack $\Comm_n$ of \eqref{eqn:commuting stack} parametrizes pairs of commuting $n \times n$ matrices modulo the general linear group $GL_n$. Similarly, the flag commuting stack $\FComm_n^{\bullet}$ of \eqref{eqn:flag commuting stack} parametrizes pairs of commuting upper triangular matrices (whose diagonals are $(x,\dots,x)$ and $(0,\dots,0)$, respectively, for any $x \in \C$) modulo the Borel subgroup $B_n$ of upper triangular matrices. There is a proper morphism: 
\begin{equation}
\label{eq: fcomm to comm}
\pi: \FComm_n^{\bullet} \to \Comm_n
\end{equation}
 Any stack which is a quotient by the Borel subgroup (in this case, $\FComm_n^{\bullet}$) comes endowed with line bundles $\CL_1,\dots,\CL_n$ that arise from the $n$ diagonal characters of $B_n$. Thus, for all vectors of integers $\dd = (d_1,\dots,d_n)$, we may consider:
\begin{equation}
\label{eqn:coxeter sheaf}
\CE_{\dd} = \pi_* (\CL_1^{d_1} \dots \CL_n^{d_n}) \in D^b_{\C^* \times \C^*}(\Coh(\Comm_n))
\end{equation}
Let us consider the categorical version of the Hall product studied by Schiffmann-Vasserot \cite{SV}, namely the bifunctor:
\begin{equation}
    \label{eqn:hall product}
D^b_{\C^* \times \C^*}(\Coh(\Comm_n)) \otimes D^b_{\C^* \times \C^*}(\Coh(\Comm_k)) \stackrel{\star}\rightarrow D^b_{\C^* \times \C^*}(\Coh(\Comm_{n+k}))
\end{equation}
given by the formula:
$$
\alpha \star \beta = Rp_{1*} \left(Lp_2^*(\alpha \boxtimes \beta) \right)
$$
Above, we let $\Comm_{n,k}$ be the stack parametrizing pairs of commuting $(n,k)$-block upper triangular matrices, modulo simultaneous conjugation by the natural maximal parabolic subgroup of $GL_{n+k}$, and consider the maps:
$$
\begin{tikzcd}
& \Comm_{n,k}  \arrow[ld, swap,"p_1"] \arrow[rd, "p_2"] & \\
\Comm_{n+k}  & & \Comm_n \times \Comm_k 
\end{tikzcd}
$$
which remember the various diagonal blocks of the matrices parametrized by $\Comm_{n,k}$. By repeated applications of the operation \eqref{eqn:hall product} to the objects \eqref{eqn:coxeter sheaf}, we may construct the following objects for any collection of integers \eqref{eqn:vector of integers}:
\begin{equation}
\label{eqn:general sheaf}
\CE_{\dd^1} \star \dots \star \CE_{\dd^r} \in D^b_{\C^* \times \C^*}(\Coh(\Comm_n))
\end{equation}

\bigskip

\noindent \textbf{We expect the functor \eqref{eqn:functor} to intertwine the products \eqref{eqn:unconstructed product} and \eqref{eqn:hall product}, and to send:
$$
E_{\dd} \text{ of \eqref{eqn:coxeter trace}} \quad \mapsto \quad \CE_{\dd} \text{ of \eqref{eqn:coxeter sheaf}} 
$$
As a consequence, the functor \eqref{eqn:functor} would send the objects \eqref{eqn:general trace} to \eqref{eqn:general sheaf}.}

\bigskip

\subsection{} Let us reformulate the main result of \cite{Hecke}, which inspired  the present paper.

\begin{theorem}[\cite{Hecke}]
\label{thm: hecke}
For any $\dd=(d_1,\ldots,d_n) \in \Z^n$ and $k \in \Z$, there exists a collection of objects $g_0, \dots, g_n \in D^b_{\C^* \times \C^*}(\Coh(\Comm_n))$ with the following properties:
\begin{itemize}
    \item $g_0=\CE_{(k)} \star \CE_{\dd}$ and $g_n=\CE_{\dd} \star \CE_{(k)}$ 
    \item For all $i\in \{1,\ldots,n\}$ there exist morphisms:
    $$\begin{cases}
    g_{i-1}\rightarrow g_i & \text{if}\ k \geq d_i\\
    g_{i-1}\leftarrow g_i & \text{if}\ k \leq d_i
    \end{cases}
    $$
    which are mutually inverse isomorphisms if $k = d_i$.
    
    \item The cones of morphisms in the previous bullet are filtered by
    \begin{equation}
    \label{eqn:cone geom}
    [\C_{q_2} \xrightarrow{0} \C_1] \cdot \begin{cases}
    \CE_{(d_1,\ldots,d_{i-1},k-a,d_{i}+a,d_{i+1},\ldots,d_n)},\ 1\le a\le k-d_i &  \text{if}\ k>d_i\\
    \CE_{(d_1,\ldots,d_{i-1},d_i-a,k+a,d_{i+1},\ldots,d_n)},\ 1\le a\le d_i-k &  \text{if}\ k<d_i
    \end{cases}
    \end{equation}
    
\end{itemize}

\noindent In the formula above, $\C_{\chi}$ denotes the one-dimensional vector space, viewed as a representation of $\C^* \times \C^*$ through the character $\chi$, and we write $q_1$ and $q_2$ for the elementary characters of the two factors of $\C^*$. The asymmetry between $q_1$ and $q_2$ in \eqref{eqn:cone geom} is due to the fact that $\FComm_n^{\bullet}$ parametrizes pairs of commuting matrices, the first of which is allowed to have scalar diagonal, but the second of which is forced to have zero diagonal.

\end{theorem}

There are certain differences between the formulation of Theorem \ref{thm: hecke} and that of Theorem 1.1 of \cite{Hecke}, which we will now explain. In the language of \emph{loc. cit.} (applied to the situation of the surface $S = \A^2$), the symbols $\CE_{(d_1,\dots,d_n)}$ should be interpreted as functors:
$$
\D^b_{\C^* \times \C^*}(\Coh(\Hilb_k)) \longrightarrow \D^b_{\C^* \times \C^*}(\Coh(\Hilb_{k+n} \times \A^1))
$$
for all $k \in \N$, where $\Hilb_k$ denotes the Hilbert scheme of $k$ points on $\A^2$, and $\A^1 \hookrightarrow \A^2$ denotes the $x$-axis. If one composes the functor above with push-forward along $\A^1$, one obtains functors:
$$
\D^b_{\C^* \times \C^*}(\Coh(\Hilb_k)) \longrightarrow \D^b_{\C^* \times \C^*}(\Coh(\Hilb_{k+n}))
$$
which satisfy the relations postulated in the three bullets in Theorem \ref{thm: hecke}. The match between this result and the formulation in Theorem \ref{thm: hecke} is completed by the correspondence:
\begin{multline}
\label{eqn:correspondence}
\Big\{\text{objects in }D^b_{\C^* \times \C^*}(\text{Coh}(\Comm_n)) \Big\} \quad \leadsto \\ \Big\{ \text{functors } \D^b_{\C^* \times \C^*}(\Coh(\Hilb_k)) \rightarrow \D^b_{\C^* \times \C^*}(\Coh(\Hilb_{k+n})) \text{ for all }k \in \N \Big\}
\end{multline}
that underlies the philosophy of \cite{SV}. Under the correspondence above, the categorified Hall product $\star$ on the left-hand side corresponds to composition of functors on the right-hand side. Of course, the discussion above is not a proof that Theorem \ref{thm: hecke} follows from Theorem 1.1 of \cite{Hecke}, but it indicates the means through which to adapt the language of \emph{loc. cit.} to establish Theorem \ref{thm: hecke} (see \cite{PZ} for an example of this principle).

\begin{remark}

In fact, the results of \cite{Hecke} yield an explicit geometric construction of the objects $g_i$, and morphisms between them. In the present paper, we will mirror this construction for objects in the category $\CK(\ASBim_n)$.

\end{remark}

\subsection{Results}
We are ready to describe our main results. First, we describe the generators of the dg category $\Tr(\ASBim_n)$.

\begin{theorem}
\label{thm: gens}
The objects $E_{\dd^1}\star \cdots \star E_{\dd^r}$ of \eqref{eqn:general trace}, as $\dd^1,\dots,\dd^r$ run over all collections of integers \eqref{eqn:vector of integers}, generate $\Tr(\ASBim_n)$. That is, any object of $\Tr(\ASBim_n)$
can be presented as a bounded twisted complex with entries given by finite direct sums of $E_{\dd^1}\star \cdots \star E_{\dd^r}$ and their shifts.
\end{theorem}

Next, we describe some interesting morphisms and relations between these generators.

\begin{theorem}
\label{thm:rel 1}
(Theorem \ref{lem: cone alpha})
For all integers $d_1,\dots,d_n$ and all $i \in \{1,\dots,n-1\}$, there is a map: 
\begin{equation}
\label{eqn:intro map}
E_{(d_1,\ldots,d_i,d_{i+1},\ldots,d_n)}\to E_{(d_1,\ldots,d_{i}-1,d_{i+1}+1,\ldots,d_n)}
\end{equation}
with the cone filtered by two copies of: 
$$
E_{(d_1,\ldots,d_i)}\star E_{(d_{i+1},\ldots,d_n)}.
$$
\end{theorem}

When the objects $E_{\dd}$ of $\Tr(\ASBim_n)$ are replaced by $\CE_{\dd}$ of $D^b_{\C^* \times \C^*}(\text{Coh}(\Comm_n))$ (see Theorem \ref{thm: hecke}), the result of Theorem \ref{thm:rel 1} was proved in Proposition 3.9 of \cite{Hecke} (in the language of functors from $\Hilb_k$ to $\Hilb_{n+k} \times \A^1$, which can be easily adapted to the language of objects on $\Comm_n$ via correspondence \eqref{eqn:correspondence}).

\begin{theorem}
\label{thm:rel 2}
(Theorem \ref{thm: one step commutation}) For any $\dd = (d_1,\dots,d_n) \in \Z^n$ and $k \in \Z$, there exists a collection of objects $G_0,\dots,G_n \in \ASBim_n$ with the following properties:

\begin{itemize}
    
    \item $\emph{Tr}(G_0)=E_{(k)} \star E_{\dd}$ and $\emph{Tr}(G_n)=E_{\dd} \star E_{(k)}$ 
    
    \item For all $i\in \{1,\ldots,n\}$ there exist chain maps in $\CK(\ASBim_n)$:
    $$
    \begin{cases}
    G_{i-1}\xrightarrow{\varphi_i} G_i & \text{if}\ k \geq d_i\\
    G_{i-1}\xleftarrow{\overline{\varphi}_i} G_i & \text{if}\ k \leq d_i
    \end{cases}
    $$
    which are mutually inverse isomorphisms if $k = d_i$.
    
    \item $\emph{Tr}(\emph{Cone}(\varphi_i))$ and $\emph{Tr}(\emph{Cone}(\overline{\varphi}_i))$ are filtered by:
    \begin{equation}
    \label{eqn:cone top}
    [\C_{q} \xrightarrow{0} \C_{q^{-1}}] \cdot \begin{cases}
    E_{(d_1,\ldots,d_{i-1},k-a,d_{i}+a,d_{i+1},\ldots,d_n)},\ 1\le a\le k-d_i &  \text{if}\ k>d_i\\
    E_{(d_1,\ldots,d_{i-1},d_i-a,k+a,d_{i+1},\ldots,d_n)},\ 1\le a\le d_i-k &  \text{if}\ k<d_i
    \end{cases}
    \end{equation}
    respectively. Above, $q$ denotes the internal grading on Soergel bimodules.
    
\end{itemize}

\end{theorem}

By comparing Theorem \ref{thm: hecke} with Theorem \ref{thm:rel 2}, we may summarize our expectations about Problem \ref{prob:main} as follows. Note that the equivariant parameters $q_1$ and $q_2$ must be specialized to $q^{-2}$ and $q^2$, respectively, and we need to multiply $E_{(d_1,\dots,d_n)}$ by $q^{1-n}$ in order for relations \eqref{eqn:cone top} to perfectly match \eqref{eqn:cone geom}.

\begin{conjecture}
\label{conj: main}
The functor \eqref{eqn:functor} should:
\begin{itemize}
    \item[(a)] send $E_{\dd^1}\star \cdots \star E_{\dd^r}$ to $\CE_{\dd^1}\star \cdots \star \CE_{\dd^r}$
    \item[(b)] send $\Tr(G_i)$ to $g_i$, in the notation of Theorems \ref{thm: hecke} and \ref{thm:rel 2}
    \item[(c)] send the traces of the morphisms between the $G_i$'s constructed in Theorem \ref{thm:rel 2} to the morphisms between the $g_i$'s constructed in Theorem \ref{thm: hecke}.
\end{itemize}
\end{conjecture}

Theorem \ref{thm: gens} and part (a) of the conjecture immediately imply the following:

\begin{corollary}
\label{cor:intro}
Assuming Conjecture \ref{conj: main}, the essential image of the functor \eqref{eqn:functor} is a certain subcategory of $D^b_{\C^* \times \C^*}(\Coh(\Comm_n))$ generated by the objects $\CE_{\dd^1}\star \cdots \star \CE_{\dd^r}$. 
\end{corollary}

\subsection{$K$-theory and the Elliptic Hall algebra}

We will now summarize the consequences of the results and conjectures in the previous Subsection to the level of Grothendieck groups (which we denote by $G$ instead of $K_0$ throughout the paper). In particular, the functor  \eqref{eqn:functor} should yield a linear map:
\begin{equation}
\label{eqn:functor k-theory}
G(\Tr(\ASBim_n)) \longrightarrow K_{\C^*}(\Comm_n)
\end{equation}
where the right-hand side is $K$-theory equivariant with respect to the anti-diagonal subtorus $\C^* \hookrightarrow \C^* \times \C^*$ (in other words, $q_1q_2 = 1$; the reason for this is that $q_1q_2$ measures the homological degree on the category $\CK(\ASBim_n)$, and this grading must be killed in order for the functor \eqref{eqn:functor} to descend to the Grothendieck groups). Corollary \ref{cor:intro} would imply that the image of \eqref{eqn:functor k-theory} should coincide with the linear span $\CA_n$ of the $K$-theory classes of the complexes of sheaves \eqref{eqn:general sheaf}. Moreover:
$$
\CA = \bigoplus_{n=0}^{\infty} \CA_n
$$
is an algebra under the $K$-theoretic version of the Hall product \eqref{eqn:hall product}, and the map \eqref{eqn:functor k-theory} should intertwine this multiplication with that of the yet-unconstructed \eqref{eqn:unconstructed product}. By definition, $\CA$ is the $\Z[q_1^{\pm 1}, q_2^{\pm 1}]$-algebra generated by the elements:
\begin{equation}
\label{eqn:generators of eha}
[\CE_{\dd} ] \in K_{\C^* \times \C^*}(\Comm_n)
\end{equation}
as $\dd$ runs over $\Z^n$. In Section \ref{sec:eha}, we will recall a computational tool called the shuffle algebra, which could in principle allow one to describe the full set of relations between the generators \eqref{eqn:generators of eha}. However, in practice these relations are quite complicated, and we will contend with defining a pre-quotient:
\begin{equation}
    \label{eqn:pre-quotient}
\widetilde{\CA} \twoheadrightarrow \CA 
\end{equation}
which is generated by symbols $E_{\dd}$, but modulo the simpler relations \eqref{eqn:rel a 1}--\eqref{eqn:rel a 2}.  In particular, \eqref{eqn:rel a 2} is a $K$-theoretic version of Theorem \ref{thm: hecke}.

As a consequence of Theorems \ref{thm:rel 1} and \ref{thm:rel 2}, we obtain the following.
\begin{theorem}
\label{thm: main}
There exists a algebra homomorphism:
$$
\widetilde{\CA}\Big|_{(q_1,q_2)\to (q^{-2},q^{2})} \longrightarrow \bigoplus_{n=0}^{\infty} G(\Tr(\ASBim_n))
$$
\end{theorem}

The specialization of parameters $(q_1,q_2)\to (q^{-2},q^{2})$ is a consequence of the complicated matching of gradings of both sides which was already visible in \cite{GNR} for finite Hecke categories. The category $\ASBim_n$ and its trace are naturally graded (by so-called $q$-grading), so the corresponding homotopy categories are {\em bigraded} by the $q$-grading and the homological grading. This gives the Grothendieck group $G(\Tr(\ASBim_n))$ the structure of a module over $\Z[q^{\pm 1}]$. 

On the other hand, the derived category of $\Comm_n$ is {\em triply graded} by the homological and  two equivariant gradings. The corresponding Grothendieck group, as explained above, is then a module over $\Z[q_1^{\pm 1},q_2^{\pm 1}]$. 

The conjectured functor between the two derived categories mixes the gradings in a complicated way: the homological grading on the affine Hecke side matches a linear combination of the homological and equivariant gradings on the geometric side. Thus, on the level of the Grothendieck group we are required to collapse one of the equivariant gradings.

\begin{remark}
It is expected that one could equip $\Tr(\ASBim_n)$ with an additional {\em third grading} which would map the geometric side. In this case, one should be able to state Theorem \ref{thm: main} without specialization of equivariant parameters. 
\end{remark}

It is natural to conjecture that the homomorphism in Theorem \ref{thm: main} factors through the quotient \eqref{eqn:pre-quotient}. If true, then we would dare to expect that the functor \eqref{eqn:functor k-theory} is an isomorphism onto its image $\CA$.

\begin{remark}

We note that $\widetilde{\CA}$ and $\CA$ are different only as $\Z[q_1^{\pm 1}, q_2^{\pm 1}]$-algebras, but we will actually show in Section \ref{sec:eha} that they are isomorphic once we tensor with $\Q(q_1,q_2)$ (upon this localization, they will also be isomorphic to the elliptic Hall algebra $\CE_{q_1,q_2}$ of \cite{BS}, thus tying everything in with \eqref{eqn:ms}). So relations \eqref{eqn:rel a 1} and \eqref{eqn:rel a 2} describe the full ideal of relations among the elements \eqref{eqn:generators of eha} over the field $\Q(q_1,q_2)$. Of course, the operation of tensoring with $\Q(q_1,q_2)$ does not make much sense at the categorical level.

\end{remark}

Alternatively, we can interpret Theorem \ref{thm: main} as a statement about the cocenter of type $A$ affine Hecke algebras, which is defined as: 
\begin{equation}
\label{eqn:trace algebra}
\Tr(\AH_n)=\AH_n/[\AH_n,\AH_n].
\end{equation}
The results of He and Nie \cite{HN} give an explicit basis of $\Tr(\AH_n)$ (see Theorem \ref{th: trace hecke}). However, to the best of our knowledge the multiplicative structure of the algebra $\bigoplus_{n} \Tr(\AH_n)$ has not been yet described in this basis.

\begin{theorem}
\label{thm:surj intro}
(Theorem \ref{thm:surj}) There is a surjective algebra homomorphism: 
$$
\widetilde{\CA}\Big|_{(q_1,q_2)\rightarrow (q^{-2},q^{2})}\to \bigoplus_{n=0}^{\infty} \Tr(\AH_n)
$$
compatible with the natural maps:
$$
\Tr(\AH_n)=\Tr(G(\ASBim_n))\to G(\Tr(\ASBim_n)).
$$
\end{theorem}

\subsection{Comparison with Oblomkov-Rozansky theory}
\label{subsec:OR}

In a series of papers \cite{OR1,OR2,OR3} Oblomkov and Rozansky constructed yet another categorification of the affine Hecke algebra using a sophisticated monoidal category of matrix factorizations which we will denote by $\MF_n$. It is expected to be closely related to the the geometric realization of the affine Hecke algebra using the derived category of the Springer resolution, constructed by Bezrukavnikov and Riche \cite{Bez,BezR}.
In this short section we review the relation between their results and this paper.

The main result of \cite{OR1} constructs a homomorphism from $\ABr_n$ to $\MF_n$: to every affine braid one can associate an object of $\MF_n$, the product of braids corresponds to the product in $\MF_n$, and the braid relations are satisfied. In particular, there are analogues of the objects $Y^{\dd_1}T_{\cox_{n_1}} \star \cdots\star Y^{\dd^r} T_{\cox_{n_r}}$ in $\MF_n$. Moreover, there is a natural (dg) functor: 
$$
F: \MF_n\to D^b_{\C^* \times \C^*}(\Coh(\Comm_n))
$$
The main result of \cite{OR2} shows that $F(\text{analogue of }Y^{\dd_1}T_{\cox_{n_1}} \star \cdots\star Y^{\dd^r} T_{\cox_{n_r}})=\CE_{\dd^1} \star \dots \star \CE_{\dd^r}$, confirming the analogue of Conjecture \ref{conj: main}(a) in Oblomkov-Rozansky theory.

To the best of our knowledge, it is not known whether the categories  $\MF_n$ and $\ASBim_n$ are equivalent (for their finite analogues see \cite{OR3}), or whether the functor $F$ is trace-like in the sense of \cite{GHW}. Nevertheless, one can ask whether $\Tr(\MF_n)$ is generated by the traces of the analogues of the objects $Y^{\dd_1}T_{\cox_{n_1}} \star \cdots\star Y^{\dd^r} T_{\cox_{n_r}}$, or whether the analogues of Theorems \ref{thm:rel 1} and \ref{thm:rel 2}, or Conjectures \ref{conj: main}(b,c) hold in $\Tr(\MF_n)$. 

\subsection{Acknowledgments} 

We thank Roman Bezrukavnikov, Ben Elias, Matt Hogancamp, Ivan Losev, Alexei Oblomkov, Lev Rozansky, Kostya Tolmachov, Paul Wedrich and Yu Zhao for useful discussions. E. G. was partially supported by the NSF FRG grant DMS-1760329. A. N. was supported by the NSF grants DMS-1760264 and DMS-1845034, as well as the Alfred P. Sloan Foundation and the MIT Research Support Committee.

\section{An algebra of many faces}
\label{sec:eha}

\subsection{Motivation from geometry} The main purpose of the present paper is to consider the category of affine Soergel bimodules $\ASBim_n$ (which will be defined in Section \ref{sec:category}) and describe its horizontal trace. To get a feel of what the answer should look like, we recall that the horizontal trace of the category of finite Soergel bimodules $\SBim_n$ was computed in \cite{GHW}, where it was shown that:
\begin{equation}
\label{eqn:equivalence sbim}
\Tr(\SBim_n) \ \cong \ D^b_{\C^* \times \C^*}(\text{Coh}(\text{Hilb}_n))
\end{equation}
In the right-hand side, we encounter the derived category of $\C^* \times \C^*$ equivariant coherent sheaves on the Hilbert scheme of $n$ points on $\mathbb{A}^2$. For general reasons that have been explored in \cite{OR1,OR2,OR3}, one would therefore expect a connection between:
\begin{equation}
\label{eqn:equivalence asbim}
\Tr(\ASBim_n) \quad \text{and} \quad D^b_{\C^* \times \C^*}(\text{Coh}(\text{Comm}_n))
\end{equation}
where in the right-hand side, we encounter the commuting stack:
\begin{equation}
\label{eqn:commuting stack}
\text{Comm}_n = V/GL_n 
\end{equation}
In the formula above, we set:
\begin{equation}
\label{eqn:def of v}
V = \Big\{ X,Y \in \text{Mat}_{n \times n}, \ [X,Y] = 0 \Big\}
\end{equation}
and define the action $GL_n \curvearrowright V$ is by simultaneous conjugation. The action $\C^* \times \C^* \curvearrowright V$ by rescaling the $X$ and $Y$ matrices commutes with the action of $GL_n$, and thus descends to an action $\C^* \times \C^* \curvearrowright \text{Comm}_n$, which defines the right-hand side of \eqref{eqn:equivalence asbim}. 

\begin{remark}
\label{rem:dg comm}

The reason we do not claim \eqref{eqn:equivalence asbim} as an equivalence is that, as opposed from the well-understood category in the right-hand side of \eqref{eqn:equivalence sbim}, the derived category of coherent sheaves on the commuting stack is very hard to describe. Moreover, one should not consider $V$ as a singular affine scheme, but as the dg scheme cut out by the equations:
\begin{equation}
\label{eqn:equations}
\sum_{a=1}^n (x_{ia} y_{aj} - y_{ia} x_{aj}) = 0, \qquad \forall \ i,j \in \{1,\dots,n\}
\end{equation}
from affine space with coordinates $\{x_{ij}, y_{ij}\}_{1\leq i,j \leq n}$. Thus, $\Comm_n$ should be interpreted as a dg stack, and its derived category is defined accordingly. More precisely, the collection of equations \eqref{eqn:equations} is a section of the standard $\mathfrak{gl}_n$-bundle over the $GL_n$-equivariant $2n^2$ dimensional affine space that parametrizes the entries of the matrices $X$ and $Y$.

\end{remark}

\subsection{K-theory} Things become a bit more manageable if we pass to the Grothendieck groups of the categories in \eqref{eqn:equivalence asbim}, so our concrete aim is to understand the relation between: 
\begin{equation}
\label{eqn:conjecture asbim}
G(\Tr(\ASBim_n)) \quad \text{and} \quad K_{\C^* \times \C^*}(\text{Comm}_n)
\end{equation}
While still not completely understood, we can shed some light on the right-hand side of \eqref{eqn:conjecture asbim} using the language of shuffle algebras, as follows. We have the closed embedding $j : \text{point} \hookrightarrow V$ of the point $X=Y=0$. The pull-back morphism takes the form:
\begin{multline}
\label{eqn:restriction}
K_{\C^* \times \C^*}(\text{Comm}_n) = K_{GL_n \times \C^* \times \C^*}(V) \xrightarrow{j^!} \\ \rightarrow K_{GL_n \times \C^* \times \C^*}(\text{point}) = \Z[q_1^{\pm 1}, q_2^{\pm 1}][z_1^{\pm 1}, \dots, z_n^{\pm 1}]^{\text{sym}}
\end{multline}
where $q_1,q_2$ are the inverses of the standard characters of $\C^* \times \C^*$, $z_1, \dots, z_n$ are characters of an arbitrary maximal torus of $GL_n$, and the superscript $\text{sym}$ denotes symmetric functions in the variables $z_1,\dots,z_n$. The localization theorem implies that, if we consider $j^!$ over the field $\Q(q_1,q_2)$, then we obtain an isomorphism:
$$
K_{\C^* \times \C^*}(\text{Comm}_n) \bigotimes_{\Z[q_1^{\pm 1}, q_2^{\pm 1}]} \Q(q_1,q_2)  \cong \Q(q_1,q_2)[z_1^{\pm 1}, \dots, z_n^{\pm 1}]^{\text{sym}}
$$
As shown in \cite{VV}, $K_{\C^* \times \C^*}(\text{Comm}_n)$ is a free $\Z[q_1^{\pm 1}, q_2^{\pm 1}]$-module, so we have an injection:
$$
K_{\C^* \times \C^*}(\text{Comm}_n) \stackrel{\tilde{\iota}}\hookrightarrow \Q(q_1,q_2)[z_1^{\pm 1}, \dots, z_n^{\pm 1}]^{\text{sym}}
$$

\begin{proposition}
\label{prop: zhao wheel}
(\cite{Zhao2}) The image of the map $\tilde{\iota}$ lands in the vector subspace:
$$
\Shuf_n \subset \Q(q_1,q_2)[z_1^{\pm 1}, \dots, z_n^{\pm 1}]^{\emph{sym}}
$$
of symmetric Laurent polynomials $F(z_1,\dots,z_n)$ which satisfy the wheel conditions (\cite{FHHSY}):
\begin{equation}
\label{eqn:wheel}
F(x,xq_1,xq_1q_2,z_4,\dots,z_n) = F(x,xq_2,xq_1q_2,z_4,\dots,z_n) = 0
\end{equation}

\end{proposition}

In light of Proposition \ref{prop: zhao wheel} above, we conclude that we have a map:
\begin{equation}
\label{eqn:iota}
K_{\C^* \times \C^*}(\text{Comm}_n) \stackrel{\iota}\hookrightarrow \Shuf_n
\end{equation}
The take-away of the present Subsection is that the right-hand side of \eqref{eqn:conjecture asbim} is a certain $\Z[q_1^{\pm 1}, q_2^{\pm 1}]$-integral form of the $\Q(q_1,q_2)$-vector space $\Shuf_n$, a claim which uses Theorem 2.5 of \cite{Shuf} to establish the fact that $\iota$ becomes an isomorphism upon tensoring with $\Q(q_1,q_2)$.

\subsection{The shuffle algebra} There is an interesting convolution product (\cite{SV2}) on:
\begin{equation}
\label{eqn:k}
K = \bigoplus_{n=0} K_{\C^* \times \C^*}(\text{Comm}_n)
\end{equation}
which is none other than the $K$-theoretic shadow of \eqref{eqn:hall product}. Meanwhile, we may consider the so-called shuffle algebra:
$$
\Shuf = \bigoplus_{n = 0}^{\infty} \Shuf_n
$$
which is endowed with the following shuffle product (\cite{FHHSY}):
\begin{equation}
\label{eqn:shuf prod}
F(z_1,\dots,z_n) * F'(z_1,\dots,z_{n'}) = 
\end{equation}
$$
= \text{Sym} \left[ \frac {F(z_1,\dots,z_n)F'(z_{n+1},\dots,z_{n+n'})}{n!n'!} \prod_{1\leq i \leq n < j \leq n+n'} \frac {\left(1 - \frac {z_iq_1}{z_j} \right)\left(1 - \frac {z_iq_2}{z_j} \right)\left(1 - \frac {z_jq_1q_2}{z_i} \right)}{1 - \frac {z_i}{z_j}} \right]
$$
where $\text{Sym}$ denotes symmetrization over all $(n+n')!$ permutations of the variables. The maps \eqref{eqn:iota} give rise to an algebra homomorphism:
$$
K \stackrel{\iota}\hookrightarrow \Shuf
$$
It is clear from the discussion above that the shuffle algebra $\Shuf$ provides a good computational model for the $K$-theory groups of commuting stacks. 

\subsection{Constructing subalgebras} The power of the shuffle algebra is that it allows one to construct many interesting elements. For example, for any $\dd = (d_1,\dots,d_n) \in \Z^n$ the element:
\begin{equation}
\label{eqn:r d}
R_{\dd} = (1-q_1)^{n-1} (1-q_2)^n \cdot 
\end{equation}
$$
\text{Sym} \left[ \frac {z_1^{d_1}\dots z_n^{d_n}}{\left(1 - \frac {z_2 q_1q_2}{z_1} \right) \dots \left(1 - \frac {z_n q_1q_2}{z_{n-1}} \right)} \right. \\ \left. \prod_{1 \leq i < j \leq n} \frac {\left(1 - \frac {z_iq_1}{z_j} \right)\left(1 - \frac {z_iq_2}{z_j} \right)\left(1 - \frac {z_jq_1q_2}{z_i} \right)}{1 - \frac {z_i}{z_j}} \right]
$$
clearly lies in $\Shuf_n \cap \Z[q_1^{\pm 1}, q_2^{\pm 1}][z_1^{\pm 1}, \dots, z_n^{\pm n}]^{\text{sym}}$.

\begin{definition}
\label{def:subalg}

Consider the $\Z[q_1^{\pm 1}, q_2^{\pm 1}]$-subalgebra $\CA \subset \Shuf$ generated by $R_{\dd}$, as $\dd \in \Z^n$.

\end{definition}

\begin{proposition}
\label{prop:contains}

We have $\CA \subseteq \emph{Im }\iota$, so we may think of $\CA$ as a subalgebra of $K$.

\end{proposition}

\begin{proof} The proof is an adaptation of the main construction of \cite{Hecke}. Specifically, consider the following version of the flag commuting stack:
\begin{equation}
\label{eqn:flag commuting stack}
\text{FComm}^\bullet_n = U/B_n
\end{equation}
where:
\begin{equation}
\label{eqn:def of u}
U = \left\{X = \begin{pmatrix} x & * & \dots & * \\ 0 & x & \dots & * \\ \vdots & \vdots & \ddots & \vdots \\ 0 & 0 & \dots & x \end{pmatrix}, Y = \begin{pmatrix} 0 & * & \dots & * \\ 0 & 0 & \dots & * \\ \vdots & \vdots & \ddots & \vdots \\ 0 & 0 & \dots & 0 \end{pmatrix}, \ [X,Y] = 0 \right\}
\end{equation}
(the number $x$ and the over-diagonal entries are arbitrary) and the subgroup $B_n \subset GL_n$ of upper triangular matrices acts on $U$ by simultaneous conjugation. Strictly speaking, $U$ is defined as the dg scheme cut out by the equations:
$$
\sum_{a = i+1}^{j-1} (x_{ia} y_{aj} - y_{ia} x_{aj}) = 0, \qquad \forall \ i+1 < j \in \{1,\dots,n\}
$$
from the affine space with coordinates $\{x,x_{ij}, y_{ij}\}_{1 \leq i < j \leq n}$. There is a natural map:
$$
\text{FComm}^\bullet_n \xrightarrow{\pi} \text{Comm}_n
$$
induced by the inclusion of the set of pairs of upper triangular matrices in the set of pairs of all matrices. Then Proposition \ref{prop:contains} follows from the claim that for all $d_1,\dots,d_n \in \Z$:
\begin{equation}
\label{eqn:claim}
R_{(d_1,\dots,d_n)} = \iota(\pi_*(\CL^{d_1}_1 \dots \CL_n^{d_n}))
\end{equation}
where the line bundles $\CL_1, \dots, \CL_n$ on $\text{FComm}_n^\bullet$ are induced from the standard diagonal characters of the Borel subgroup $B_n$. Equality \eqref{eqn:claim} is a straightforward computation involving push-forward and pull-back morphisms; we will explain the main idea of the proof, and leave the details as an exercise to the interested reader (see \cite[Proposition 2.7]{Wheel} for a closely related computation). Since $\iota$ is induced by the pull-back morphism \eqref{eqn:restriction}, one computes the right-hand side of \eqref{eqn:claim} as the succession of the following steps:

\begin{itemize}[leftmargin=*]

\item Push-forward the $B_n$-equivariant line bundle $\CL_1^{d_1} \dots \CL_n^{d_n}$ from $U$ to $V$ 

\item Restrict to the origin, yielding a Laurent polynomial in the characters $z_1,\dots,z_n$ of $B_n$

\item Induce from a $B_n$-equivariant coherent sheaf to a $GL_n$-equivariant coherent sheaf

\end{itemize}

One can switch the order of the first two operations, at the cost of multiplying with the following Euler class right after the second step:
\begin{equation}
\label{eqn:euler class}
\prod_{1 < i+1 < j \leq n} \left(1 - \frac {z_jq_1q_2}{z_i} \right)
\end{equation}
This is due to the fact that we impose fewer quadratic equations on $U$ than on $V$. With this in mind, let us start from the $B_n$ character of the restriction of the line bundle $\CL_1^{d_1} \dots \CL_n^{d_n}$ to the origin of $U$, namely $z_1^{d_1} \dots z_n^{d_n}$. We must multiply this monomial by:
$$
\prod_{1 \leq i < j \leq n} \left[ \left(1 - \frac {z_iq_1}{z_j} \right)\left(1 - \frac {z_iq_2}{z_j} \right)\right]
$$
which is due to the fact that the lower-diagonal entries are missing when defining $U$ but are present when defining $V$. As explained above, we should also multiply the answer by \eqref{eqn:euler class}. Finally, the third bullet above involves dividing by the Weyl denominator:
$$
\prod_{1\leq i < j \leq n} \left(1 - \frac {z_i}{z_j} \right)
$$
and symmetrizing. The result of these operations is precisely $R_{(d_1,\dots,d_n)}$ of \eqref{eqn:r d}. 

\end{proof}

\subsection{Relations in $\CA$} The following relation in $\Shuf$ is easy to prove:
\begin{equation}
\label{eqn:rel shuf} 
R_{(d_1,\dots,d_i,d_{i+1},\dots,d_n)} - q_1 q_2 \cdot R_{(d_1,\dots,d_i-1,d_{i+1}+1,\dots,d_n)} = (1-q_1) R_{(d_1,\dots,d_i)} R_{(d_{i+1},\dots,d_n)}
\end{equation}
for all $d_1,\dots,d_n\in \Z$ and $i \in \{1,\dots,n-1\}$. Therefore, if we define for all $(m,n) \in \Z \times \N$:
$$
H_{m,n} = R_{(d_1,\dots,d_n)} \quad \text{where} \quad d_i = \left \lfloor \frac {mi}n \right \rfloor - \left \lfloor \frac {m(i-1)}n \right \rfloor 
$$
it is easy to conclude from \eqref{eqn:rel shuf} that $\CA$ is generated by the $H_{m,n}$'s as an algebra. Indeed, using \eqref{eqn:rel shuf} one can express all $R_{\dd}$ via $H_{|\dd|,n}$ (where $|\dd| = d_1+\dots+d_n$) and products of smaller $R_{\dd'}$. Moreover, it was shown in \cite{W} that for all $(m,n), (m',n') \in \Z \times \N$ we have:
\begin{equation}
\label{eqn:comm}
[H_{m,n}, H_{m',n'}] = \mathop{\sum^{n_1+\dots+n_t = n}_{\frac {m_1}{n_1} \leq \dots \leq \frac {m_t}{n_t}}}^{m_1+\dots+m_t = m} \text{coefficient} \cdot H_{m_1,n_1} \dots H_{m_t,n_t}
\end{equation}
for various coefficients in $\Z[q_1^{\pm 1}, q_2^{\pm 1}]$. Using this relation, one obtains the PBW theorem:
\begin{equation}
\label{eqn:pbw}
\CA = \mathop{\bigoplus^{n_1+\dots+n_t = n}_{\frac {m_1}{n_1} \leq \dots \leq \frac {m_t}{n_t}}}^{m_1+\dots+m_t = m} \Z[q_1^{\pm 1}, q_2^{\pm 1}] \cdot H_{m_1,n_1} \dots H_{m_t,n_t}
\end{equation}
(since $H_{m,n}$ commutes with $H_{m',n'}$ if $\frac mn = \frac {m'}{n'}$, it does not matter how we order $H_{m,n}$'s with the same slope in \eqref{eqn:pbw}). Our interest in the subalgebra $\CA$ is motivated by the following.

\begin{conjecture}
\label{conj:iso}

There is an isomorphism $\CA\Big|_{(q_1,q_2)\to (q^{-2},q^{2})} \cong \bigoplus_{n=0}^\infty G(\Tr(\ASBim_n))$. 

\end{conjecture}

\subsection{Understanding relations} Unless we had a way to connect $G(\Tr(\ASBim_n))$ to the shuffle algebra directly, it seems like one's best bet to prove Conjecture \ref{conj:iso} is to understand $\CA$ by generators and relations. The most direct way is to think of the $H_{m,n}$'s as generators and \eqref{eqn:comm} as relations, but the latter are not explicit and so difficult to check in the trace of the category $\ASBim_n$. The notion below is the ``next best thing".

\begin{definition}
\label{def:a}

(\cite{Hecke}) Consider the algebra: 
\begin{equation}
\label{eqn:a}
\widetilde{\CA} = \Z[q_1^{\pm 1}, q_2^{\pm 1}] \Big \langle \EE_{(d_1,\dots,d_n)} \Big \rangle_{n \in \N, d_1,\dots,d_n \in \Z} \Big/ \text{relations \eqref{eqn:rel a 1}--\eqref{eqn:rel a 2}}
\end{equation}
where for all $d_1,\dots,d_n\in \Z$ and $i \in \{1,\dots,n-1\}$ we set:
\begin{equation}
\label{eqn:rel a 1}
\EE_{(d_1,\dots,d_i,d_{i+1},\dots,d_{n})} - q_1q_2 \cdot \EE_{(d_1,\dots,d_i-1,d_{i+1}+1,\dots,d_{n})} = (1-q_1) \EE_{(d_1,\dots,d_i)} \EE_{(d_{i+1},\dots,d_{n})}
\end{equation}
and for all $d_1,\dots,d_n,k \in \Z$ we set:
\begin{equation}
\label{eqn:rel a 2}
\Big[ \EE_{(k)}, \EE_{(d_1,\dots,d_n)} \Big] = (q_2-1) \sum_{i=1}^n \begin{cases} \sum_{a = 1}^{k-d_i} \EE_{(d_1,\dots,d_{i-1},k-a,d_i+a,d_{i+1},\dots,d_n)} &\text{if } k \geq d_i \\ - \sum_{a = 1}^{d_i-k} \EE_{(d_1,\dots,d_{i-1},d_i-a,k+a,d_{i+1},\dots,d_n)} &\text{if } k \leq d_i \end{cases}
\end{equation}
\end{definition}

\begin{proposition}
\label{prop:surj}

There is a surjective algebra morphism $\widetilde{\CA} \twoheadrightarrow \CA$ induced by $\EE_{\dd} \mapsto R_{\dd}$.

\end{proposition}

\begin{proof} The fact that \eqref{eqn:rel a 1} and \eqref{eqn:rel a 2} hold with $\EE$'s replaced by $R$'s was proved in \cite{Hecke} (the former of these is precisely \eqref{eqn:rel shuf}, which is a simple exercise). 
\end{proof} 

\begin{remark} 

The fact that relation \eqref{eqn:rel a 2} holds with $\EE$'s replaced by $R$'s was used in \cite{BHMPS} as a step in their proof of the Extended Delta Conjecture.

\end{remark} 

With this in mind, we are ready to state the main result of the present paper. 

\begin{theorem}
\label{thm:main}

There is a surjective homomorphism: 
$$
\widetilde{\CA} \Big|_{(q_1,q_2)\to (q^{-2},q^{2})}  \twoheadrightarrow \bigoplus_{n=0}^\infty G(\Tr(\ASBim_n)).
$$

\end{theorem}

\subsection{Comparing $\CA$ with $\widetilde{\CA}$} In the remainder of the present Section, we will explain why Theorem \ref{thm:main} is a good approximation to Conjecture \ref{conj:iso}. After all, at first glance, it seems like the algebra $\widetilde{\CA}$ only captures ``some" of the relations which hold in the algebra $\CA$. However, we will now show that all relations which are not thus captured are actually contained in the $\Q(q_1,q_2)$-torsion. More precisely, we will prove the following result.

\begin{proposition}
\label{prop:iso loc}

If we consider the localized algebra: 
$$
\CA_{\emph{loc}} = \CA \bigotimes_{\Z[q_1^{\pm 1}, q_2^{\pm 1}]} \Q(q_1,q_2)
$$
(and analogously with $\widetilde{\CA}$ replaced by $\CA$), then the map:
$$
\widetilde{\CA}_{\emph{loc}} \rightarrow \CA_{\emph{loc}}
$$
induced by Proposition \ref{prop:surj} is an isomorphism. 

\end{proposition}

The proof of the result above will occupy the remainder of the present Section. We will need to define one more algebra, which is known in the literature as the positive half of the Ding-Iohara-Miki/quantum toroidal $\mathfrak{gl}_1$ algebra.

\begin{definition}
\label{def:quantum toroidal}
Consider the algebra:
\begin{equation}
\label{eqn:quantum toroidal}
\CU = \Z[q_1^{\pm 1}, q_2^{\pm 1}] \Big \langle E_m\Big \rangle_{m \in \Z} \Big/ \text{relations \eqref{eqn:rel tor 1}--\eqref{eqn:rel tor 2}}
\end{equation}
where we consider the formal series $E(z) = \sum_{m \in \Z} \frac {E_m}{z^m}$ and let:
\begin{equation}
\label{eqn:rel tor 1} 
E(z) E(w) (z - wq_1)(z - wq_2) \left(z - \frac w{q_1q_2} \right) = E(w) E(z) (zq_1 - w)(zq_2 - w) \left(\frac z{q_1q_2} - w \right) 
\end{equation}
and:
\begin{equation}
\label{eqn:rel tor 2}
[[E_{m+1},E_{m-1}],E_m]  = 0
\end{equation}
for all $m \in \Z$.

\end{definition}

\begin{proposition}

The assignment $\{E_m \mapsto \EE_{(m)}\}_{m \in \Z}$ yields an algebra homomorphism:
$$
\CU \rightarrow \widetilde{\CA}
$$

\end{proposition}

\begin{proof} One needs to check that relations \eqref{eqn:rel tor 1} and \eqref{eqn:rel tor 2} hold in the algebra $\widetilde{\CA}$, which was done in \cite[Proposition 4.8]{Hecke}. \end{proof}

\subsection{Comparing localizations} 

Let $\CU_{\text{loc}} = \CU \bigotimes_{\Z[q_1^{\pm 1}, q_2^{\pm 1}]} \Q(q_1,q_2)$. Passing all our algebras and homomorphisms to localizations, we obtain $\Q(q_1,q_2)$-algebra homomorphisms:
\begin{equation}
\label{eqn:chain}
\CU_{\text{loc}} \xrightarrow{f} \widetilde{\CA}_{\text{loc}} \xrightarrow{g} \CA_{\text{loc}} \subseteq \Shuf
\end{equation}
While localization doesn't usually preserve injections, the latter inclusion in \eqref{eqn:chain} holds after localization because \eqref{eqn:pbw} implies that $\CA$ is a free $\Z[q_1^{\pm 1}, q_2^{\pm 1}]$-module. 

\begin{proposition}
\label{prop:1}

The map $f$ is surjective. 

\end{proposition}

\begin{proof} We must show that for all $d_1,\dots,d_n \in \Z$, the element:
$$
\EE_{(d_1,\dots,d_n)} \in \widetilde{\CA}_{\loc}
$$
lies in the subalgebra $\widetilde{\CA}^\circ_{\loc} \subseteq \widetilde{\CA}_{\loc}$ generated by the elements $\{\EE_{(m)}\}_{m \in \Z}$. We will prove this statement by induction of $n$, the base case being trivial. Because of \eqref{eqn:rel a 1} and the induction hypothesis, we have:
\begin{equation}
\label{eqn:inc}
\EE_{(d_1,\dots,d_i,d_{i+1},\dots,d_n)} - q_1q_2 \cdot \EE_{(d_1,\dots,d_i-1,d_{i+1}+1,\dots,d_n)} \in \widetilde{\CA}^\circ_{\loc}
\end{equation}
Iterating this relation yields:
\begin{equation}
\label{eqn:inc int}
\EE_{(d_1,\dots,d_n)} - (q_1q_2)^{\sum_{i=1}^n (n-i)d_i} \cdot \EE_{(0,\dots,0,d_1+\dots+ d_n)} \in \widetilde{\CA}^\circ_{\loc}
\end{equation}
Thus, it suffices to show that (we write $0^k$ for the $k$-tuple $(0,\dots,0)$):
$$
\EE_{(0^{n-1},m)} \in \widetilde{\CA}^\circ_{\loc}
$$
for all $m \in \Z$. Assuming first that $m < 0$, relation \eqref{eqn:rel a 2} gives us:
\begin{equation}
\label{eqn:inc comm}
\Big[ \EE_{(0^{n-2},m)}, \EE_{(0)} \Big] = (1-q_2) \sum_{a = m}^{-1} \EE_{(0^{n-2},a,m-a)}
\end{equation}
which by \eqref{eqn:inc int} equals:
$$
(1-q_2)\left((q_1q_2)^{-1}+\dots +(q_1q_2)^{m}\right) \EE_{(0^{n-1},m)} + \text{something in }\widetilde{\CA}^\circ_{\loc}
$$
Since the left-hand side of \eqref{eqn:inc comm} lies in $\widetilde{\CA}^\circ_{\loc}$ by the induction hypothesis, so does the right-hand side and we are done. The case $m > 0$ is treated analogously, so we leave it as an exercise to the interested reader. Finally, when $m = 0$, we use instead of \eqref{eqn:inc comm} the relation:
\begin{equation}
\label{eqn:inc comm 0}
\Big[ \EE_{(-1,0^{n-2})}, \EE_{(1)} \Big]  = (1-q_2) \Big( \EE_{(0^n)} + \sum_{i=0}^{n-2} \EE_{ (-1,0^i,1,0^{n-2-i})} \Big)
\end{equation}
By \eqref{eqn:inc int}, the right-hand side of the expression above equals:
$$
(1-q_2)\left(1+(q_1q_2)^{-1}+\dots +(q_1q_2)^{-n+1}\right) \EE_{(0,\dots,0)} + \text{elements in }\widetilde{\CA}^\circ_{\loc}
$$
Since the left-hand side of \eqref{eqn:inc comm 0} lies in $\widetilde{\CA}^\circ_{\loc}$ by the induction hypothesis, so does the right-hand side and we are done.

\end{proof}

\subsection{Extended algebras}

For any $k \in \N$, it is easy to show that there exist derivations:
\begin{align*}
&\partial_k : \widetilde{\CA}_{\loc} \longrightarrow \widetilde{\CA}_{\loc}, \qquad \partial_k\left(\EE_{(d_1,\dots,d_n)}\right) = \sum_{i=1}^n \EE_{(\dots,d_{i-1},d_i+k,d_{i+1}\dots)} \\
&\partial_k : \CA_{\loc} \longrightarrow \CA_{\loc}, \qquad \partial_k\left(R_{(d_1,\dots,d_n)}\right) = \sum_{i=1}^n R_{(\dots,d_{i-1},d_i+k,d_{i+1}\dots)} \\
&\partial_k : \Shuf \longrightarrow \Shuf, \qquad \qquad \partial_k\left(F(z_1,\dots,z_n)\right) =  F(z_1,\dots,z_n)(z_1^k+\dots+z_n^k)
\end{align*}
(for the first of these, it is a matter of showing that the derivation in question preserves the ideal generated by relations \eqref{eqn:rel a 1}--\eqref{eqn:rel a 2}, while for the latter two, it is a matter of showing that the derivations in question preserve the vector subspace of wheel conditions \eqref{eqn:wheel}; both checks are elementary, and we leave them to the interested reader). Moreover, $\partial_k$ and $\partial_l$ commute for all $k,l \in \N$, which allows us to define algebra structures on:
$$
X^{\text{ext}} = X \bigotimes_{\Q(q_1,q_2)} \Q(q_1,q_2)[H_1,H_2,\dots]
$$
for all $X \in \{\widetilde{\CA}_{\loc}, \CA_{\loc}, \Shuf\}$, by imposing the commutation relations:
$$
[x,H_k] = \partial_k(x) \qquad \forall x \in X
$$
Similarly, let us define the algebra:
$$
\CU_{\text{loc}}^{\text{ext}} = \Q(q_1,q_2) \Big \langle E_m, H_k \Big \rangle_{m \in \Z} \Big/ \text{relations \eqref{eqn:rel tor 1}--\eqref{eqn:rel tor 2}  and } [E_m, H_k] = E_{m+k}
$$
Because of the obvious compatibility between the various relations above involving the $H_k$'s, it is clear that the chain of homomorphisms \eqref{eqn:chain} extends to:
$$
\CU_{\text{loc}}^{\text{ext}} \rightarrow \widetilde{\CA}_{\text{loc}}^{\text{ext}} \xrightarrow{h} \CA_{\text{loc}}^{\text{ext}} \subseteq \Shuf^{\text{ext}}
$$
The composition of the maps above is an isomorphism, by combining the results of \cite{F,S} with those of \cite{Shuf}. The fact that the first two maps are surjective (due to Propositions \ref{prop:surj} and \ref{prop:1}) imply that the map $h$ is an isomorphism. 

\begin{proof} \emph{of Proposition \ref{prop:iso loc}:} Since the algebras $\widetilde{\CA}^{\text{ext}}_{\loc}$ and $\CA^{\text{ext}}_{\loc}$ are free right-modules over $\Q(q_1,q_2)[H_1,H_2,...]$, the fact that $h$ is an isomorphism is preserved upon factoring by the ideal $(H_1,H_2,\dots)$. This yields precisely the fact that the map $g$ is an isomorphism. 

\end{proof}

\section{The Affine Hecke category and its trace}
\label{sec:category}

\subsection{The affine braid group}
\label{sec: affine braid}

Let us recall that the extended affine braid group $\ABr_n$ is generated by symbols $\sigma_0,\sigma_1,\ldots,\sigma_{n-1}$ and $\omega$ modulo relations:
$$
\sigma_i\sigma_{i+1}\sigma_i=\sigma_{i+1}\sigma_{i}\sigma_{i+1},\ \sigma_i\sigma_j=\sigma_j\sigma_i\ (|i-j|>1),\ \omega \sigma_i\omega^{-1}=\sigma_{i+1}
$$
(indices are understood modulo $n$). Note that $\sigma_1,\ldots, \sigma_{n-1}$ and $\omega$ already generate the affine braid group. The subgroup generated by $\sigma_0, \ldots, \sigma_{n-1}$ is the braid group corresponding to the affine Coxeter group $\widehat{A_{n-1}}$. We will use the following notation:
\begin{equation}
\label{eqn:notation sigma}
\sigma_{[a,b]}=\sigma_{a}\cdots \sigma_{b-1}\ \text{for}\ a\le b.
\end{equation}
The affine braid group contains the lattice generated by the pairwise commuting elements:
$$
y_i=\sigma_{i-1}^{-1}\cdots \sigma_1^{-1}\omega \sigma_{n-1}\cdots \sigma_{i},\ i=1,\ldots, n.
$$
Note that $y_{i+1}=\sigma_i^{-1}y_i\sigma_i^{-1}$. For any $\dd = (d_1,\dots,d_n) \in \Z^n$, we call the monomial: 
\begin{equation}
\label{eqn:notation y}
y^{\dd} = y_1^{d_1}\cdots y_n^{d_n}
\end{equation}
a {\em translation}, and further use the term {\em dominant translation} if $d_1\ge d_2\ge \cdots \ge d_n$. 

We will also use the product $\star: \ABr_n\times \ABr_{k}\to \ABr_{n+k}$ which sends the generators $\sigma_i,y_i$ of $\ABr_n$ to the namesake generators of $\ABr_{n+k}$ and the generators $\sigma_j,y_j$ of $\ABr_{k}$ to $\sigma_{j+n}$ and $y_{j+n}$ respectively.

\subsection{The affine symmetric group}
\label{sec: affine symmetric} 

Recall the extended affine symmetric group:
$$
\widetilde{S_n} = \ABr_n \Big/ (\sigma_i^2 = 1)_{\forall i \in \Z/n\Z}
$$
To distinguish $\widetilde{S_n}$ from the braid group, we will denote its generators by $s_i$, and the image of $\omega$ by $\pi$. The group $\widetilde{S_n}$ can be identified with the group of $n$-periodic permutations of $\Z$, that is, bijections $v:\Z\to \Z$ such that $v(i+n)=v(i)+n$. In this presentation $s_i$ swaps $i$ and $i+1$ (and is extend periodically to all $i\in \Z$) while $\pi(i)=i+1$ for all $i$. Also: 
\begin{equation}
\label{eqn:y action}
y_i(m)=\begin{cases}
m+n &\text{if}\ m \equiv i\ mod\  n\\
m & \text{otherwise},
\end{cases}
\end{equation}
so $y_1^{d_1}\cdots y_n^{d_n}(m) = m + nd_r$ where $r=m\ \mathrm{mod}\ n$.

\begin{definition}
\label{def:degree}

There is a grading (on either $\ABr_n$ or $\widetilde{S_n}$) defined multiplicatively via:
$$
\deg(\pi)=1 \qquad \text{and} \qquad \deg(s_i)=0
$$
for all $i \in \{1,\dots,n\}$.

\end{definition}

Note that $\deg(y_i) = 1$ for all $i$, which implies that:
$$
\deg(y^{\dd}) = |\dd| :=d_1 + \dots + d_n
$$
for any $\dd = (d_1,\dots,d_n) \in \Z^n$. In the presentation of elements of $\widetilde{S_n}$ as $n$-periodic permutations, the degree of $v:\Z\to \Z$ can be computed as $\deg(v)=\frac 1n\sum_{i=1}^{n}(v(i)-i)$. The subgroup of elements of degree 0 is the usual affine symmetric group:
$$
\widehat{S_n} \subset \widetilde{S_n}
$$
i.e. the subgroup generated by the $s_i$'s. It is easy to see that any element $v \in \widetilde{S_n}$ can be uniquely written as $v = \pi^k \alpha$ for some $k \in \Z$ and some $\alpha$ of degree 0. The group $\widehat{S_n}$ is a Coxeter group and has a natural notion of length ($s_0,\ldots,s_{n-1}$ all have length 1). We define
$$
\ell(v)=\ell(\pi^k\alpha)=\ell(\alpha),\ \alpha\in \widehat{S_n}.
$$
In particular, $\ell(\pi^k)=0$ for all $k$.

Given an element $w\in \widetilde{S_n}$ we can define its positive braid lift by choosing an arbitrary reduced expression and replacing each $s_i$ by $\sigma_i$ and $\pi$ by $\omega$. Clearly, for $i>1$ the element $y_i$ is not a positive braid lift of any affine permutations. However, we have the following:

\begin{lemma}
\label{lem: dominant translation}
If $d_1\ge d_2\ge \ldots\ge d_n$, then the dominant translation $y_1^{d_1}\cdots y_n^{d_n}$ is a positive braid lift of an affine permutation.
\end{lemma}

\begin{proof}
Let us prove by induction that for all $k \in \{1,\dots n\}$ we have:
$$
y_1\cdots y_k=(\omega \sigma_{n-1}\cdots \sigma_{k})^k
$$
For $k=1$ we get $y_1=\omega \sigma_{n-1}\cdots \sigma_1$ by definition. For the step of induction, we write
$$
(\omega \sigma_{n-1}\cdots \sigma_{k})^k y_{k+1}=
(\omega \sigma_{n-1}\cdots \sigma_{k})\cdots (\omega \sigma_{n-1}\cdots \sigma_{k})\sigma_k^{-1}\cdots \sigma_1^{-1}\omega \sigma_{n-1}\cdots \sigma_{k+1}.
$$
For $j<m-1$ we have
$$
(\omega \sigma_{n-1}\cdots \sigma_{m})\sigma_j^{-1}=
\omega \sigma_j^{-1}\sigma_{n-1}\cdots \sigma_{m}=
\sigma_{j+1}^{-1}(\omega \sigma_{n-1}\cdots \sigma_{m}),
$$
so
$$
(\omega \sigma_{n-1}\cdots \sigma_{k})^k y_{k+1}=(\omega \sigma_{n-1}\cdots \sigma_{k})\sigma_k^{-1}\cdots (\omega \sigma_{n-1}\cdots \sigma_{k})\sigma_k^{-1}\cdot \omega \sigma_{n-1}\cdots \sigma_{k+1}=$$
$$(\omega \sigma_{n-1}\cdots \sigma_{k+1})^{k+1}.
$$
Now 
$$
y_1^{d_1}\cdots y_n^{d_n}=(y_1\cdots y_n)^{d_n}\prod_{k=1}^{n-1}(y_1\cdots y_k)^{d_k-d_{k+1}}.
$$
\end{proof}

\subsection{The affine Hecke algebra}
\label{sec: affine hecke}

The (extended) affine Hecke algebra $\AH_n$ is yet another quotient of the affine braid group (strictly speaking, of its group algebra). It has generators $T_i$ (for all $i\in \Z/n\Z$, which correspond to $\sigma_i$) and $\Omega$ (which corresponds to $\omega$) which satisfy the additional quadratic relation:
$$
(T_i-q)(T_i+q^{-1})=0.    
$$
This relation corresponds to the skein relation, and $\AH_n$ can be identified with the skein of the annulus $\A\times [0,1]$ with $n$ marked points. Given $v\in \widetilde{S_n}$, we denote by $T_v$ the projection of its positive braid lift to $\AH_n$.  It is well known that $T_v,v\in \widetilde{S_n}$ span $\AH_n$. Furthermore, define the {\em cocenter} (or the trace) of $\AH_n$ as:
$$
\Tr(\AH_n)=\AH_n/[\AH_n,\AH_n].
$$
Topologically, this is isomorphic to the degree $n$ part of the skein of the torus. The following result gives an explicit basis in $\Tr(\AH_n)$:

\begin{theorem}[\cite{HN}]
\label{th: trace hecke}

a) Suppose that $v_1,v_2$ are conjugate in $\widetilde{S_n}$ and have minimal length in their conjugacy class. Then $[T_{v_1}]=[T_{v_2}]$ in $\Tr(\AH_n)$.

b) The cocenter $\Tr(\AH_n)$ has a basis $[T_{v_i}]$ where $v_i$ are minimal length representatives of all conjugacy classes in $\widetilde{S_n}$ (by the part (a), the choice of a minimal length representative does not matter).

\end{theorem}

Thus, to get a concrete description of $\Tr(\AH_n)$ we need to describe conjugacy classes in $\widetilde{S_n}$, which will be done in next lemmas. We will use the notation \eqref{eqn:notation y} in the affine symmetric group as well as the affine braid group, and further write $\ee_i \in \Z^n$ for the vector with a single 1 at position $i$, and zeroes everywhere else.

\begin{lemma}
\label{conj: coxeter}
Two affine permutations of the form $y^{\dd}s_1\cdots s_{n-1}\in \widetilde{S_n}$ are conjugate if and only if they have the same degree $|\dd|$.

\end{lemma}

\begin{proof}
Clearly, conjugate permutations have the same degree. Conversely, we can write:
$$
y^{\dd}s_1\cdots s_{n-1}\in \widetilde{S_n}\sim y^{\dd-\ee_i}s_1\cdots s_{n-1}y_i=
$$
$$
y^{\dd-\ee_i}s_1\cdots s_{i-1}(s_{i}y_i)\cdots s_{n-1}=y^{\dd-\ee_i}s_1\cdots s_{i-1}(y_{i+1}s_i)\cdots s_{n-1}=y^{\dd-\ee_i+\ee_{i+1}}s_1\cdots s_{n-1}.
$$
Therefore we can change the vector of exponents  $\dd$ to $\dd-\ee_i+\ee_{i+1}$ without changing the conjugacy class. By using these operations, we can relate any two vectors with the same degree by conjugations.
\end{proof}

\begin{corollary}
\label{cor: dmn}
Suppose that $|\dd|=m$ and $\gcd(m,n)=1$. Then $y^{\dd}s_1 \dots s_{n-1}$ is conjugate to $\pi^m$, which is the minimal length representative in its conjugacy class.
\end{corollary}

\begin{proof}
Let $mi=q_{i}n+r_i$ where $1\le i\le n$ and $0\le r_i\le n-1$. Then $\pi^m(x)=x+m$, so for $1\le i<n$ we get
$$
\pi^m(r_i)=\pi^m(mi-q_{i}n)=m(i+1)-q_{i}n=(q_{i+1}-q_{i})n+r_{i+1}.
$$
Define an affine permutation $\phi$ such that $\phi(i)=r_i$ for $1\le i\le n$ (it is well defined since $\gcd(m,n)=1$), then for $1\le i<n$:
$$
\phi^{-1}\pi^m\phi(i)=\phi^{-1}\pi^m(r_i)=\phi^{-1}((q_{i+1}-q_i)n+r_{i+1})=(q_{i+1}-q_i)n+(i+1),
$$
For $i=n$ we get $r_n=0$, so
$
\phi^{-1}\pi^m\phi(n)=\phi^{-1}\pi^m(0)=\phi^{-1}(m)=q_1n+1,
$
and 
$$\phi^{-1}\pi^m\phi=y_1^{q_1}y_2^{q_2-q_1}\cdots y_{n}^{q_ {n}-q_{n-1}}\cdot s_1 \dots s_{n-1}.$$
By Lemma \ref{conj: coxeter} we conclude that $\pi^m$ is conjugate to $y^{\dd}s_1 \dots s_{n-1}$ with 
$$
|\dd|=q_1+(q_2-q_1)+\ldots+(q_n-q_{n-1})=m.
$$
\end{proof}

\begin{lemma}
\label{conj: general}
Two affine permutations $y^{\dd}w$ and $y^{\dd'}w'\ (w,w'\in S_n)$ are conjugate in $\widetilde{S_n}$ if and only if the following two conditions hold:
\begin{itemize}
\item[(a)] $w$ and $w'$ have the same cycle type, and
\item[(b)] the cycles in $w$ and $w'$ can be matched such that for each cycle $(i_1\dots i_k)$ in $w$ that corresponds to a cycle $(j_1 \dots j_k)$ in $w'$, we have $d_{i_1}+\dots+d_{i_k} = d'_{j_1} + \dots + d'_{j_k}$.

\end{itemize}
\end{lemma}

\begin{proof}
The conjugation by a finite permutation reorders the components in $\dd$ and conjugates $w$, so the condition (a) is clearly necessary. On the other hand, similarly to Lemma \ref{conj: coxeter} we can write:
$$
y_i^{-1}(y^{\dd}w)y_i=y_i^{-1}y^{\dd}y_{w(i)}w=y^{\dd-\ee_i+\ee_{w(i)}}w,
$$
so conjugation by $y_i$ fixes the sum $d_{i_1}+\dots+d_{i_k}$ for any cycle $(i_1 \dots i_k)$ of $w$, and any two affine permutations with thus matching sums are conjugate to each other.
\end{proof}

 For a sequence of vectors $\dd^1,\ldots,\dd^r$, we define an affine permutation
$$
e_{\dd^1}\star\cdots \star e_{\dd^r}=e_{(d_1,\dots,d_{n_1})}\star e_{(d_{n_1+1},\dots, d_{n_1+n_2})}\star \dots\star e_{(d_{n-n_r+1},\dots, d_{n})}=$$
$$y_1^{d_1} \dots y_{n}^{d_{n}}\sigma_{[1,n_1]}\cdot \sigma_{[n_1+1,n_1+n_2]}\dots\sigma_{[n-n_r+1,n]}
$$
Lemma \ref{conj: general} implies that any affine permutation is conjugate to some $e_{\dd^1}\star\cdots \star e_{\dd^r}$.

\begin{definition}(\cite{HN})
\label{def: newton}
Let $v\in \widetilde{S_n}$ be an affine permutation. If $v^k=y^{\dd}$ for some $\dd=(d_1,\ldots,d_n)$, we define the {\em Newton point} of $v$ by the equation
$
\nu(v):=\frac{1}{k}\dd.
$
\end{definition}

It is easy to see  that this definition does not depend on the choice of $k$. Also, conjugating $v$ corresponds to permuting the components of $\nu(v)$.

\begin{lemma}
\label{lem: newton}
If $v=e_{\dd^1}\star \cdots \star e_{\dd^r}$ where $\dd^i$ has length $n_i$, then: 
$$
\nu(v)=\left(\underbrace{\frac{|\dd^1|}{n_1},\ldots,\frac{|\dd^1|}{n_1}}_{n_1},\ldots,\underbrace{\frac{|\dd^r|}{n_r},\ldots,\frac{|\dd^r|}{n_r}}_{n_r}\right)
$$
\end{lemma}

\begin{proof}
First observe that if $u$ and $v$ commute and $u^{k}=y^{\dd},v^{k'}=y^{\dd'}$ then:
$$
(uv)^{kk'}=y^{k'\dd+k\dd'} \quad \Rightarrow \quad \nu(uv)=\frac{k'\dd+k\dd'}{kk'}=\nu(u)+\nu(v),
$$
so it is sufficient to prove the statement for a single $e_{\dd}$. In this case
$$
(e_{\dd})^n=(y_1^{d_1}\cdots y_n^{d_n}s_1 \dots s_{n-1})^{n}=y_1^{|\dd|}\cdots y_n^{|\dd|}\quad 
\Rightarrow \quad  
\nu(e_{\dd})=\left(\frac{|\dd|}{n},\ldots,\frac{|\dd|}{n}\right).
$$
\end{proof}

\begin{corollary}
\label{cor: convex paths}
The conjugacy classes in $\widetilde{S_n}$ are in bijection with convex paths on the plane. These label both the PBW basis \eqref{eqn:pbw} in the elliptic Hall algebra and the basis in $\Tr(\AH_n)$ by Theorem \ref{th: trace hecke}.
\end{corollary}
 
\begin{proof}
By conjugating $v$, we can assume that $y^{\dd}$ that appears in Definition \ref{def: newton} is dominant. For the particular $v$'s that appear in Lemma \ref{lem: newton}, this means that up to conjugating $v$, the sequence $\frac{|\dd^i|}{n_i}$ can be made non-decreasing. This corresponds to a convex path on the plane with steps $(n_i,|\dd^i|)$, and by Lemma \ref{conj: general} such path uniquely determines a conjugacy class in $\widetilde{S_n}$.
\end{proof}

\subsection{The affine Hecke category}
\label{sec: affine sbim}

Following Elias and Mackaay-Thiel \cite{Elias,MT}, we consider the type $A$ extended affine Hecke category $\ASBim_n$, which is defined as follows. Let $R=\C[x_1,\ldots,x_n]$ and $\tR=\C[x_1,\ldots,x_n,\delta]$, and note that the symmetric group $S_n$ acts on both $R$ and $\tR$ by permuting $x_i$ and fixing $\delta$. The rings $R$ and $\tR$ are graded such that the variables $x_i$ and $\delta$ have grading 2.
We have an additional endomorphism of $\tR$ given by: 
$$
\pi(\delta)=\delta,\ \pi(x_n)=x_1-\delta,\ \pi(x_i)=x_{i+1},\ 1\le i\le n-1.
$$
It is easy to see that $\pi$ and $S_n$ define an action of $\widetilde{S_n}$ on $\tR$. We will consider $R-R$ (respectively $\tR-\tR$) bimodules, the simplest being $\one = R$ (respectively $\tR$) with the usual left and right multiplication. We will also encounter the  Bott-Samelson bimodules:
$$
\overline{B_i}=R\otimes_{R^{s_i}}R,\qquad  B_i=\tR\otimes_{\tR^{s_i}}\tR 
$$
There is an additional $(\tR,\tR)$-bimodule $\Omega$ which is isomorphic to $\tR$, where the left action of $\tR$ is standard and the right action is twisted by $\pi$. One can check that  $\Omega \otimes B_i\otimes \Omega^{-1}\simeq B_{i+1}$.

\begin{definition}
\label{def:asbim}

The category of finite Soergel bimodules $\SBim_n$ is defined as the smallest full subcategory of $R-R$ bimodules containing $R$ and $\overline{B_i}$ and closed under tensor products, direct sums and direct summands. Similarly, the category of extended affine Soergel bimodules $\ASBim_n$ is defined as a smallest full subcategory of $\tR-\tR$ bimodules containing $R,B_i$ and $\Omega$ and closed under tensor products, direct sums and direct summands.

\end{definition}

Note that the subcategory of  $\ASBim_n$ generated by $R$ and $B_i$ is equivalent to  $\SBim_n$ with scalars extended from $R$ to $\tR$. As in  \cite{Rouquier}, one can define Rouquier complexes: 
$$
T_i=q^{-1}[B_i\to \one], \qquad T_i^{-1}=q^{-1}[q^2\one\to B_i]
$$
which satisfy the braid relations up to homotopy ($q$ records a grading shift). It is easy to see that $\Omega T_i\Omega^{-1}=T_{i+1}$, so the assignment $\sigma_i \mapsto T_i, \omega \mapsto \Omega$ induces a homomorphism from the extended affine braid group to the homotopy category $\CK(\ASBim_n)$.

Given an affine permutation $v$, we can consider its positive braid lift and the corresponding Rouquier complex $T_v$. Clearly, $T_{\omega^k\alpha}=\Omega^k T_{\alpha}$. 
Elias proved in \cite{Elias} that:
$$
\Hom(\Omega^k T_{\alpha},\Omega^{k'}T_{\beta})=\begin{cases}
\Hom(T_{\alpha},T_{\beta}) & \text{if}\ k=k'\\
0 & \text{otherwise}.
\end{cases}
$$
Since by Lemma \ref{lem: dominant translation} dominant translations are positive lifts of certain affine permutations, we may assign to them Rouquier complexes. One can extend these to all translations by presenting them as ratios of dominant ones in an arbitrary way.

\begin{definition}

For all $i \in \Z/n\Z$, let $Y_i \in \ASBim_n$ be the Rouquier complex corresponding to the affine braid $y_i$. For any $\dd = (d_1,\dots,d_n)\in \Z^n$, consider:
$$
Y^{\dd} = Y_1^{d_1}\cdots Y_n^{d_n} \in \ASBim_n
$$
which are called {\em Wakimoto objects} in \cite{Elias}. 

\end{definition}

An easy consequence of the definitions is that $T_i^{-1}Y_iT_i^{-1}\simeq Y_{i+1}$. 
The automorphisms of Rouquier complexes are particularly easy.

\begin{lemma}
\label{lem: rouquier auto}
(a) For any affine braid $v\in \ABr_n$ we have:
$$
\Hom(T_{v},T_{v})\simeq \Hom(\one,\one)=\widetilde{R}.
$$
(b) On any $T_v$, the left action of any $f\in \widetilde{R}$ is homotopic to the right action of $v^{-1}(f) \in \widetilde{R}$. 
\end{lemma}
\begin{proof}
(a) Since Rouquier complexes are invertible in $\CK(\ASBim_n)$, we have $\Hom(T_{v},T_{v})\simeq \Hom(\one,\one)$, and the latter is clearly $\widetilde{R}$. 

(b) It suffices to consider the cases $T_v = T_i$ (where the left action of $x_i$ is homotopic to the right action of $x_{i+1}$) and $T_v = \Omega$ (where the left and right actions are equal up to a twist by $\pi$). 
\end{proof}
\begin{corollary}
\label{cor: rouquier auto}
The left action of $x_i$ on $Y_1^{d_1}\cdots Y_n^{d_n}T_w,\ w\in S_n$ is homotopic to the right action of $x_{w^{-1}(i)-d_i\delta}$.
\end{corollary}
Unless stated otherwise, we will use left action of $\widetilde{R}$ to parametrize the automorphisms of Rouquier complexes, but implicitly we will use Lemma \ref{lem: rouquier auto} to translate it into the right action if necessary.

\subsection{Generators of $\ASBim_n$} We record a useful categorification of the quadratic relation in the Hecke algebra. There is a chain map from $T_i$ to $T_i^{-1}$, and its cone can be written as:
\begin{equation}
\label{eq: skein 0}
[T_i\to T_i^{-1}]\simeq [q\one \xrightarrow{x_i-x_{i+1}} q^{-1}\one].
\end{equation}
Equivalently, there is a chain map from $T_i^2$ to $\one$ such that:
\begin{equation}
\label{eq: skein}
[T_i^2\to \one]\simeq [qT_i\xrightarrow{x_i-x_{i+1}}q^{-1} T_i].
\end{equation}
We say that a (possibly infinite) collection of objects $X_1,X_2,\ldots$ generates a pre-triangu-lated dg category $\CS$ if any object $\CS$ can be presented as a finite twisted complex (or iterated cone) built out of finite direct sums of the $X_i$'s and their grading shifts. 

\begin{theorem}[\cite{Elias}]
\label{thm: rouquier generate}
The Rouquier complexes $T_v,\ v\in \widetilde{S_n}$ generate $\CK(\ASBim_n)$ as a pre-triangulated dg category. Combining this with \eqref{eq: skein 0}, we conclude that:
$$
G(\ASBim_n)=G(\CK(\ASBim_n))=\AH_n
$$
with the classes of $T_v$ in the Grothendieck group corresponding to the namesake generators of $\AH_n$.
\end{theorem}

\begin{proposition}
\label{prop:generate}
The category $\CK(\ASBim_n)$  is generated  by the Rouquier complexes $Y^{\dd} T_w$, for all $\dd \in \Z^n$ and $w\in S_n$. 
\end{proposition}

\begin{proof}
By Theorem \ref{thm: rouquier generate}, the category $\CK(\ASBim_n)$ is generated by the products of $\Omega$ and $T_i^{\pm 1}, 1\le i\le n-1$. We need to show that any product of this form can be resolved by $Y^{\dd} T_w$. Since we can write $\Omega=Y_1T_{1}^{-1}\cdots T_{n-1}^{-1}$, then any product of $\Omega$ and $T_i^{\pm 1}$ can be rewritten as a product of $Y_i$ and $T_i^{\pm 1}$. Now we use the relations:
$$
T_{i-1}Y_i=Y_{i-1}T_{i-1}^{-1},\ T_i^{-1}Y_i=Y_{i+1}T_i,\ T_jY_i=Y_iT_j,\quad j\neq i,i+1
$$
and \eqref{eq: skein 0} to move all the $T_i^{\pm 1}$'s to the right (for example, we may resolve $T_iY_i$ by $T_i^{-1}Y_i=Y_{i+1}T_i$ and $Y_i$). As a result, we can resolve any object of $\CK(\ASBim_n)$ by Rouquier complexes of the form $Y^{\dd}T_{\beta}$ where $T_{\beta}$ is the Rouquier complex for a finite braid $\beta\in \Br_n$, so in particular $T_{\beta}\in \CK(\SBim_n)$. It remains to use the well-known fact that $\CK(\SBim_n)$ is generated by $T_{w}, w\in S_n$.

\end{proof}

\subsection{The horizontal trace}
\label{sec: trace}

Let $\CS$ be a monoidal category. 

\begin{definition}(\cite{BHLZ,BPW}) The {\em horizontal trace} $\Tr=\Tr(\CS)$ of $\CS$ is the category whose objects are the same as those of $\CS$, and whose morphisms are defined as follows:
\begin{equation}
\label{eqn:hom trace}
\Hom_{\Tr}(X,Y)=\bigoplus_{Z \in \emph{Ob } \CS}\Hom(ZX,YZ) \ / \sim
\end{equation}
where the equivalence relation identifies the compositions: 
\begin{equation}
\label{eq: trace relation}
ZX\xrightarrow{f}WX\xrightarrow{g}YZ\qquad \text{and} \qquad WX\xrightarrow{g} YZ\xrightarrow{f} YW
\end{equation}
for arbitrary maps $f:Z\to W$ and $g:WX\to YZ$.  

\end{definition}
Although the trace has the same objects as the original category, its morphism sets are quite different: on the one hand there are more morphisms because of the direct sum in \eqref{eqn:hom trace}, on the other hand there are fewer morphisms because of the equivalence relation \eqref{eq: trace relation}.

We will refer to \eqref{eq: trace relation} as the {\em trace relation}. Note that it does not define a dg category structure on $\Tr(\CS)$, even if $\CS$ was a dg category to begin with. In fact, \eqref{eq: trace relation} can be understood as the truncation of the more complicated {\em derived trace} of $\CS$ from \cite{GHW}, but the former notion will be enough for the purposes of this paper. Furthermore, even if $\CS$ is (pre)-triangulated, $\Tr(\CS)$ might not be closed under cones of morphisms. Below we will always take the pre-triangulated hull of $\Tr(\CS)$ and denote it the same way, see \cite{GHW} for more details. 
With this definition we get  $\Tr(\CS)=\Tr(\CK(\CS))$ \cite[Lemma 6.25]{GHW}.

There is a trace functor $\Tr:\CS\to \Tr(\CS)$ which sends an object $X$ to the namesake object in $\Tr(\CS)$. If $\CS$ has duals then 
$\Tr(XY)\simeq \Tr(YX)$ for all $X,Y\in \CS$. Therefore we have a commutative diagram:
\begin{equation}
\label{eq: K0 tr}
\begin{tikzcd}
G(\CS)  \arrow{d} \arrow{r} & G(\Tr(\CS))\\
\Tr(G(\CS)) \arrow{ur} & 
\end{tikzcd}
\end{equation}
where the bottom-most $\Tr$ is defined as in \eqref{eqn:trace algebra}. We henceforth set $\CS = \CK(\ASBim_n)$ and abbreviate $\Tr(\CS) = \Tr(\ASBim_n)$. In this case, the diagram \eqref{eq: K0 tr} becomes:
\begin{equation*}
\begin{tikzcd}
\AH_n  \arrow{d} \arrow{r} & G(\Tr(\ASBim_n))\\
\Tr(\AH_n) \arrow{ur} & 
\end{tikzcd}
\end{equation*}

First, we record two easy consequences of the definitions.

\begin{proposition}
We have $\Hom_{\Tr}(X,Y)=0$ unless $X$ and $Y$ have the same degree (the grading on objects of $\ASBim_n$ is inherited from the grading of Definition \ref{def:degree}).
\end{proposition}

\begin{proof}
The space $\Hom_{\Tr}(X,Y)$ is generated by $\Hom(T_{w}X,YT_{w})$, which vanishes unless: 
$$
\deg(T_wX)=\deg(T_w)+\deg(X)=\deg(T_w)+\deg(Y)=\deg(YT_w),
$$
so $\deg(X)=\deg(Y)$.
\end{proof}

\begin{lemma}
\label{lem: full twist}
For all $X$ and $Y$ in $\ASBim_n$ we have 
$$
\Hom_{\Tr}(X,Y)=\Hom_{\Tr}(\Omega^nX,\Omega^nY).
$$
\end{lemma}

\begin{proof}
The object $\Omega^n$ is in the Drinfeld center of the original category $\ASBim_n$ (\cite{Elias}), that is, $X\Omega^n\simeq \Omega^nX$ for all $X$ and the isomorphism is functorial in $X$. Tensoring by an object in the Drinfeld center of $\CS$ defines an endofunctor of $\Tr(\CS)$ \cite[Section 6.6]{GHW}, and in this case it is invertible with the inverse $\Omega^{-n}$, so the result follows.

In other words, since $\Omega^n$ is central and invertible, it commutes with the trace relation \eqref{eq: trace relation} and hence preserves the morphism spaces in the trace.
\end{proof}

We also record the action of $\widetilde{R}$ on the traces of Rouquier complexes.

\begin{lemma}
\label{lem: endo trace rouquier}
For any affine braid $v\in \ABr_n$ we have a natural action of the quotient: 
$$
R_v:=\widetilde{R}\Big/(\delta,x_1-v(x_1), \dots, x_n - v(x_n))
$$
on the trace $\Tr(T_v)$. Here, as above, $v$ acts on $\widetilde{R}$ via its projection to $\widetilde{S_n}$. In particular, the action of $\delta$ on $\Tr(T_v)$ is trivial for all $v$.

\end{lemma}
 
\begin{proof}
By Lemma \ref{lem: rouquier auto} the left action of $\widetilde{R}$ on $T_v$ is homotopic to the right action twisted by $v$. The trace relation \eqref{eq: trace relation} identifies the left action and right action of $\widetilde{R}$, and hence factors through $\widetilde{R}/(x_i-v(x_i))$. It remains to prove that the action of $\delta$ vanishes.

Assume first that the Newton point $\nu(v)$ is nonzero. 
Suppose that $v$ projects to $y^{\dd}w\in \widetilde{S_n}$. By Corollary \ref{cor: rouquier auto} we have $x_i=x_{w^{-1}(i)}-d_i\delta$,
in particular, if $(y^{\dd}w)^m=y^{m\nu(v)}$ then $x_i=x_i-m\nu(v)_i\delta$. Since $\nu(v)\neq 0$, this equation forces $\delta=0$. 

Finally, assume that $\nu(v)=0.$ Then by Lemma \ref{lem: full twist} we have: 
$$
\Hom_{\Tr}(T_v,T_v)=\Hom_{\Tr}(\Omega^{n}T_v,\Omega^{n}T_v)
$$
and $\nu(\omega^{n}v)=\nu(v)+(1,\ldots,1)\neq 0$. Since $\delta$ acts by zero on $\Tr(\Omega^nT_v)$ (by the previous paragraph), we conclude that it also acts by zero on $\Tr(T_v)$.

\end{proof}

\section{Trace of the affine Hecke category: Generators}
\label{sec: skein gens}

In the remainder of the present paper we consider $\Tr(\ASBim_n)$, the horizontal trace of the affine Hecke category. Theorem \ref{thm: rouquier generate} and Proposition \ref{prop:generate} yield generators of the category $\CK(\ASBim_n)$, so their traces generate $\Tr(\ASBim_n)$. In this Section we choose a subset of these generators and prove that they still generate the trace category. In the next Section we will describe certain explicit exact sequences relating these generators. 

\subsection{Generation by $E_{\dd}$} Recall the objects $E_{\dd} \in \Tr(\ASBim_n)$ of \eqref{eqn:coxeter trace}.

\begin{lemma}
\label{lem: push through w}
Let $w$ be a permutation on the first, respectively last, $k$ strands (that is, in the subgroup of $S_n$ generated by $s_1,\ldots,s_{k-1}$, respectively $s_{n-k},\dots,s_{n-1}$).
Then for any $\dd \in \Z^n$, the object $Y^{\dd}T_w \in \CK(\ASBim_n)$ can be written as a twisted complex of $T_u Y^{\dd'}$'s, where $u$ goes over permutations on the first, respectively last, $k$ strands and $\dd'$ goes over $\Z^n$.
\end{lemma}

\begin{proof}
It suffices to prove the Lemma for the product $Y_jT_i$ for all $i < k$, respectively $i > n-k$, and $j \in \{1,\dots,n\}$. Since $T_i^{-1}Y_iT_i^{-1}\simeq Y_{i+1}$, there are three possible cases:

1) If $j\neq i,i+1$ then $Y_jT_i\simeq T_iY_j$.

2) If $j=i$ then $Y_iT_i \simeq T_iY_{i+1}T_i^2$ which is an extension between two copies of $T_iY_{i+1}T_i \simeq Y_i$ and $T_iY_{i+1}$ (see \eqref{eq: skein}).

3) If $j=i+1$ then $Y_{i+1}T_{i} \simeq T_i^{-1}Y_i$ which is an extension between $T_iY_i$ and two copies of $Y_i$  (see \eqref{eq: skein 0}).
\end{proof}

\begin{lemma}
\label{lem: trace generate}
$\Tr(\ASBim_n)$ is generated by the objects $\Tr(Y^{\dd}T_w)$ where $w$ goes over subwords of the Coxeter element $s_1\cdots s_{n-1}$ and $\dd$ goes over $\Z^n$.
\end{lemma}

\begin{proof}
By Proposition \ref{prop:generate} the horizontal trace is generated by the objects $\Tr(Y^{\dd}T_w)$ as $w$ goes over $S_n$. It is elementary to see that we can choose a reduced expression for $w$ which contains at most one copy of $s_{1}$. If it contains no copy of $s_{1}$, we proceed by induction, otherwise we write $w=\alpha s_{1}\beta$, where $\alpha$ and $\beta$ only use the last $n-1$ strands. Therefore $Y^{\dd}T_w=Y^{\dd}T_{\alpha}T_{1}T_{\beta}$, and by Lemma \ref{lem: push through w} we can resolve this object by $T_{u}Y^{\dd'}T_{1}T_{\beta}$ where $u$ only uses the last $n-1$ strands. The trace of $Y^{\dd}T_w$ can thus be resolved by $\Tr(T_{u}Y^{\dd'}T_{1}T_{\beta})\simeq \Tr(Y^{\dd'}T_{1}T_{\beta}T_{u})$ which itself can be resolved by $Y^{\dd''}T_{1}T_{w'}$ where $w'$ only uses the last $n-1$ strands and $\dd''$ runs over $\Z^n$. 

The paragraph above establishes the base case of the following inductive statement: the horizontal trace is generated by the objects $\Tr(Y^{\dd} T_c T_{w})$, where $c$ runs over subwords of $s_1 \dots s_k$ and $w$ only uses the last $n-k-1$ strands. To prove the induction step, either $w$ only uses the last $n-k-2$ strands (in which case the induction step is trivial), or we may write $T_w=T_{\alpha} T_{k+1}T_{\beta}$ where $\alpha$ and $\beta$ only use the last $n-k-2$ strands. Then $\alpha$ commutes with all subwords of $s_1 \dots s_k$, hence we can write:
$$
\Tr(Y^{\dd} T_c T_{w}) = \Tr(Y^{\dd} T_c T_{\alpha} T_{k+1}T_{\beta}) = \Tr(Y^{\dd}  T_{\alpha} T_c T_{k+1}T_{\beta})
$$
which by Lemma \ref{lem: push through w} can be resolved by:
$$ 
\Tr(T_u Y^{\dd'} T_c T_{k+1} T_{\beta})\simeq \Tr(Y^{\dd'} T_c T_{k+1} T_{\beta} T_u)
$$
where $u$ uses the last $n-k-2$ strands. Since $cs_{k+1}$ is a subword of $s_1 \dots s_{k+1}$, this establishes the induction step. Thus the proof is complete.

\end{proof}

We will now review the objects defined in \eqref{eqn:coxeter trace} and \eqref{eqn:general trace}.

\begin{definition}
For any $d_1,\dots,d_n \in \Z$, we define the object:
\begin{equation}
\label{eqn:coxeter object}
E_{(d_1,\ldots,d_n)} :=\Tr(Y_1^{d_1} \dots Y_n^{d_n}T_{1}\dots T_{n-1}) \in \Tr(\ASBim_n)
\end{equation}
More generally, given a composition $n_1+\dots+n_r = n$, we define the object:
\begin{equation}
\label{eqn:general object}
E_{(d_1,\dots,d_{n_1})}\star E_{(d_{n_1+1},\dots, d_{n_1+n_2})}\star \dots\star E_{(d_{n-n_r+1},\dots, d_{n})} = 
\end{equation}
$$
= \Tr\left(Y_1^{d_1} \dots Y_{n_1}^{d_{n_1}}T_{[1,n_1]}\cdot Y_{n_1+1}^{d_{n_1+1}}\dots Y_{n_1+n_2}^{d_{n_1+n_2}}T_{[n_1+1,n_1+n_2]}\dots 
Y_{n-n_r+1}^{d_{n-n_r+1}}\dots Y_{n}^{d_{n}}T_{[n-n_r+1,n]}\right)
$$
where $T_{[a,b]} = T_a \dots T_{b-1}$, by analogy with \eqref{eqn:notation sigma}.
\end{definition}

We will call $|\dd|=d_1+\dots+d_n$ the {\em degree} of the objects \eqref{eqn:coxeter object} and \eqref{eqn:general object}. It is easy to see that it agrees with the degree of an affine braid of Definition \ref{def:degree}.

By Lemma \ref{lem: trace generate}  the trace is generated by the objects  \eqref{eqn:general object}. By Lemma \ref{lem: endo trace rouquier} the action of $\widetilde{R}$ on $E_{(d_1,\dots,d_{n_1})}\star E_{(d_{n_1+1},\dots, d_{n_1+n_2})}\star \dots\star E_{(d_{n-n_r+1},\dots, d_{n})}$ factors through
$$
R_{n_1,\dots,n_r} = \widetilde{R} \Big / (\delta = 0,x_1=\ldots=x_{n_1},x_{n_1+1} = \dots = x_{n_1+n_2},\ldots,x_{n-n_r+1}=\ldots=x_{n})
$$
We will henceforth work only with the quotient above, and in particular, identify the left and right action of $x_i$ for all $i \in \{1, \dots, n\}$. 

\begin{remark}
Note that this quotient of $\widetilde{R}$ precisely matches the support condition for the sheaves \eqref{eqn:general sheaf} on the commuting stack. Indeed, the sheaf $\CE_{\dd}$ arises from \eqref{eqn:def of u}, with all the eigenvalues of the $X$ matrix being equal. Similarly, $\CE_{\dd^1} \star \dots \star \CE_{\dd^r}$ arises from the analogue of \eqref{eqn:def of u} where the eigenvalues of the $X$ matrix are equal in blocks of sizes $n_1$, $n_2$, \dots, $n_r$.

\end{remark}

\subsection{Generation by convex paths}

In this section we prove a categorical version of Theorem \ref{th: trace hecke}. For this, we will need some notations. 

\begin{definition}
We will call {\em strong conjugation} the transitive closure of the following relation on $\widetilde{S_n}$:
$v\approx v'$ if $\ell(v)=\ell(v')$, $v'=xvx^{-1}$ and
either $\ell(xv)=\ell(x)+\ell(v)$ or $\ell(vx^{-1})=\ell(x)+\ell(v)$.
\end{definition}

\begin{lemma}
\label{lem: strongly conjugate}
Suppose that $v$ and $v'$ are strongly conjugate ($v\approx v'$) and $T_v, T_{v'}$ are the Rouquier complexes for the positive braid lifts of $v,v'$. Then $\Tr(T_v)\simeq \Tr(T_{v'})$ and the positive braid lifts of $v$ and $v'$ are conjugate in $\ABr_n$. 
\end{lemma}

\begin{proof}
Suppose that $\ell(v)=\ell(v')$, $v'=xvx^{-1}$ and
 $\ell(xv)=\ell(x)+\ell(v)=\ell(v'x).$ Then:
$$
T_{x}T_{v}=T_{xv}=T_{v'x}=T_{v'}T_{x} \qquad \Rightarrow \qquad T_{v'}=T_{x}T_{v}T_{x}^{-1}.
$$
Similarly, if $\ell(vx^{-1})=\ell(x)+\ell(v)$ then:
$$
T_{v}T_{x^{-1}}=T_{vx^{-1}}=T_{x^{-1}v'}=T_{x^{-1}}T_{v'} \quad \Rightarrow \quad T_{v'}=T_{x^{-1}}^{-1}T_{v}T_{x^{-1}}.
$$
See also \cite[Lemma 5.1]{HN}.
\end{proof}

\begin{theorem}
\label{thm: HN}
a) The category $\Tr(\ASBim_n)$ is generated by $\Tr(T_v)$ where $v$ runs over minimal length elements in conjugacy classes in $\widetilde{S_n}$, and $T_v$ is the Rouquier complex for its positive braid lift.

b) If $v,v'$ are minimal length elements in the same conjugacy class then $\Tr(T_v)\simeq \Tr(T_{v'})$.
\end{theorem}

\begin{proof}
a) Recall that by Theorem \ref{thm: rouquier generate}, $\Tr(\ASBim_n)$ is generated by the objects $\Tr(T_v)$ for all $v\in \widetilde{S_n}$. Assume that $v$ is not of minimal length in its conjugacy class. Then by \cite[Theorem 2.10]{HN} there exists a simple reflection $s_i$ such that $\ell(s_ivs_i)=\ell(v)-2$. Let $v'=s_ivs_i$, then $T_v=T_iT_{v'}T_{i}$ and: 
$$
\Tr(T_v)=\Tr(T_iT_{v'}T_{i})\simeq \Tr(T_i^2T_{v'})
$$
which can be resolved by $\Tr(T_{v'})$ and two copies of $T_iT_{v'}=T_{s_iv'}$. Since $\ell(v'),\ell(s_iv')<\ell(v)$ we can proceed by induction on the length of $v$.

b) Suppose that $v,v'$ are two minimal length elements in the same conjugacy class. Then by \cite[Theorem 2.10]{HN} they are strongly conjugate, and by Lemma \ref{lem: strongly conjugate} we get $\Tr(T_v)\simeq \Tr(T_{v'}).$
\end{proof}

For all $(m,n) \in \Z \times \N$, define: 
$$
P_{m,n}:=E_{(d_1(m,n),\ldots,d_n(m,n))}, \qquad \text{where} \qquad d_i(m,n)=\left\lfloor\frac{mi}{n}\right\rfloor-\left\lfloor\frac{m(i-1)}{n}\right\rfloor
$$

\begin{theorem}
\label{thm: convex}
Let $v$ be a minimal length element in its conjugacy class. Then there exists a sequence of pairs $(m_1,n_1),\ldots,(m_r,n_r)$ such that: 
\begin{equation}
\label{eqn:objects}
\Tr(T_v)=P_{m_1,n_1}\star \cdots \star P_{m_r,n_r} \qquad \text{and} \qquad \frac{m_1}{n_1}\le \frac{m_2}{n_2}\le \cdots\le \frac{m_r}{n_r}
\end{equation}
Such a sequence is unique up to permutation of those $P_{(m_i,n_i)}$ of the same slope $\frac{m_i}{n_i}$.
\end{theorem}

\begin{remark}
Such sequences $(m_i,n_i)$ correspond to convex paths on the lattice $\Z^2$ which label the basis in $\Tr(\AH_n)$, see Corollary \ref{cor: convex paths}.
\end{remark}

\begin{proof}
The uniqueness follows from Corollary \ref{cor: convex paths}, since such $(m_i,n_i)$ are uniquely determined by the conjugacy class of $v$ in $\widetilde{S_n}$.

To prove existence, we use the topological interpretation of closed annular braids as in the Introduction. Let $\beta(v)\in \ABr_n$ be the positive braid lift of $v$, then $\ell(v)$ equals the number of crossings in $\beta(v)$. The closure of $\beta(v)$ is a link $L_v$ in the thickened torus with several components (representing some homology classes $(m_i,n_i)\in H^2(\mathbb{T},\Z)\simeq \Z^2$), and all crossings between different components are positive. This implies $\frac{m_i}{n_i}\le \frac{m_j}{n_j}$ for $i<j$.

Since $v$ has minimal length in its conjugacy class, each component of $L_v$ has minimal possible number $\gcd(m_i,n_i)-1$ of self-intersections in its homology class $(m_i,n_i)$. This can be seen by gluing the torus from a square, drawing $\gcd(m_i,n_i)$ parallel straight lines of slope $\frac{m_i}{n_i}$ on it and connecting them by a Coxeter braid with $\gcd(m_i,n_i)-1$ crossings. Such a curve intersects the horizontal side of the square in $n_i$ points, and hence is a closure of an annular braid on $n_i$ strands which is conjugate to $P_{m_i,n_i}$ in $\ABr_n$. See also \cite[Section 6]{Mellit}.
\end{proof}

\begin{remark}
As a warning to the reader, it is not true that $P_{m,n}$ represent positive braid lifts of permutations, they are only conjugate to them. For example, suppose that $\gcd(m,n)=1$, then $\Omega^m$ is the unique positive braid lift of a minimal length element in its conjugacy class. It has length $0$ and corresponds to the $(m,n)$ torus knot. Thus, while: 
$$
Y_1^{d_1(m,n)}\cdots Y_n^{d_n(m,n)}T_1\cdots T_{n-1}
$$
is not equal to $\Omega^m$, we are claiming that they are conjugate in the affine braid group similarly to Corollary \ref{cor: dmn}. 

For $m=1$ we get $d_i(m,n)=0$ for $i<n$ and $d_n(m,n)=1$, so we need to consider the element
$$
Y_nT_1\cdots T_{n-1}=T_{n-1}^{-1}\cdots T_1^{-1}\Omega T_1\cdots T_{n-1}
$$
which is conjugate to $\Omega$.
\end{remark}

By combining Theorems \ref{thm: HN} and \ref{thm: convex} we get the following.

\begin{corollary}
The category $\Tr(\ASBim_n)$ is generated by the objects \eqref{eqn:objects}. 

\end{corollary}

\section{Trace of the affine Hecke category: Relations}

\subsection{Exact sequences}

We begin by collecting several useful exact sequences in the affine Hecke category and its trace. 

\begin{lemma}
\label{lem: skein 1}
There are chain maps $\phi: Y_iT_i\to T_{i}Y_{i+1}$ and $\psi : T_iY_i\to Y_{i+1}T_i$, whose cones have the following form:
\begin{align}
&[Y_iT_i \xrightarrow{\phi} T_{i}Y_{i+1}]\simeq [qY_i\xrightarrow{x_{i+1}-x_i-\delta} q^{-1}Y_{i}] \label{eqn:align 1} \\
&[T_iY_i \xrightarrow{\psi} Y_{i+1}T_i]\simeq [qY_i\xrightarrow{x_i-x_{i+1}} q^{-1}Y_{i}] \label{eqn:align 2}
\end{align} 
Furthermore, $\Hom(Y_iT_i,T_{i}Y_{i+1})$ and $\Hom(Y_iT_i, T_{i}Y_{i+1})$ are free rank 1 $\widetilde{R}$-modules  spanned by $\phi$ and $\psi$ respectively.
\end{lemma}

\begin{proof}
We have $Y_i=T_{i}Y_{i+1}T_{i}$, so $Y_iT_i=T_{i}Y_{i+1}T_{i}^2$. 
Therefore
$$
\Hom(Y_iT_i,T_{i}Y_{i+1})\simeq \Hom(T_{i}Y_{i+1}T_{i}^2,T_{i}Y_{i+1})\simeq
\Hom(T_i^2,\one)\simeq \widetilde{R}.
$$
The second isomorphism follows from the fact that $T_{i}Y_{i+1}$ is invertible. By \eqref{eq: skein} there is a canonical map $T_i^2\to \one$ (which spans $\Hom(T_i^2,\one)\simeq \widetilde{R}$) with cone 
$[qT_i\xrightarrow{x_i-x_{i+1}} q^{-1}T_i]$. By tensoring it with $T_{i}Y_{i+1}$ on the left, we get the desired statement. Note that the left action of $(x_i-x_{i+1})$ on $T_i$ is homotopic to the left action of $(x_{i+1}-x_i-\delta)$ on $T_{i}Y_{i+1}\cdot T_{i}=Y_i$, due to Corollary \ref{cor: rouquier auto}. We thus conclude \eqref{eqn:align 1}.

Similarly, $T_iY_i=T_i^2Y_{i+1}T_i$, so by tensoring the canonical map $T_i^2\to \one$ with $Y_{i+1}T_i$ on the right yields \eqref{eqn:align 2}.
\end{proof}

\begin{remark}
\label{rem: skein 1 commute}
If $j\neq i,i+1$, then $Y_j$ commutes with $Y_i,Y_{i+1}$ and $T_i$. Furthermore, Rouquier canonicity \cite[Proposition 4.1]{Elias} implies that there is a commutative diagram of isomorphisms:
$$
\begin{tikzcd}
\Hom(Y_jT_i^2,Y_j)\arrow{r} \arrow{d} & \Hom(T_i^2,\one)\\
\Hom(T_i^2Y_j,Y_j) \arrow{ur} &
\end{tikzcd}
$$
where the rightward arrows use the fact that $Y_j$ is invertible, and the vertical arrow uses $Y_jT_i\simeq T_iY_j$. 

In short, we can say that $Y_j$ commutes with the skein map $T_i^2\to \one$, and by a similar argument, $Y_j$ commutes with the maps $\phi$ and $\psi$ from Lemma \ref{lem: skein 1}.
\end{remark}

\begin{theorem}
\label{lem: cone alpha}
For any $1 \leq i < n$ and all integers $d_1,\dots,d_n$, there is a chain map:
\begin{equation}
\label{eqn:cone alpha}
E_{(d_1,\dots,d_i,d_{i+1},\dots,d_n)} \rightarrow E_{(d_1,\dots,d_{i}-1,d_{i+1}+1,\dots,d_n)}
\end{equation}
with cone homotopy equivalent to: 
$$
\left[qE_{(d_1,\dots,d_i)}\star E_{(d_{i+1},\dots,d_n)}\xrightarrow{x_i-x_{n}} q^{-1}E_{(d_1,\dots,d_i)}\star E_{(d_{i+1},\dots,d_n)}\right].
$$
\end{theorem}

\begin{proof}
Let $\dd = (d_1,\dots,d_n)$. We have: 
$$
E_{\dd}=\Tr(Y^{\dd}T_1\cdots T_{n-1}) = 
\Tr(Y_iY^{\dd-\ee_i}T_1\cdots T_{n-1}) \simeq
$$
$$
\simeq \Tr(Y^{\dd-\ee_i}T_1\cdots T_{n-1}Y_i)=
\Tr(Y^{\dd-\ee_i}T_1\cdots T_{i-1}(T_{i}Y_i)T_{i+1}\cdots T_{n-1}).
$$
By Lemma \ref{lem: skein 1} there is a chain map 
from $E_{\dd}$ to:
$$
\Tr(Y^{\dd-\ee_i}T_1\cdots T_{i-1}(Y_{i+1}T_i)T_{i+1}\cdots T_{n-1}) = E_{\dd-\ee_i+\ee_{i+1}}
$$
The cone of this map is isomorphic to two copies of: 
$$
\Tr(Y^{\dd-\ee_i}T_1\cdots T_{i-1}\one(Y_{i})T_{i+1}\cdots T_{n-1})=
\Tr(Y^{\dd-\ee_i}T_1\cdots T_{i-1} \one T_{i+1}\cdots T_{n-1}Y_i) \simeq
$$
$$
\simeq \Tr(Y_iY^{\dd-\ee_i}T_1\cdots T_{i-1}\one T_{i+1}\cdots T_{n-1}) =
E_{(d_1,\ldots,d_i)}\star E_{(d_{i+1},\ldots,d_n)}.
$$
By Lemma \ref{lem: skein 1}, the connecting map between the aforementioned two copies is equal to $x_i-x_{i+1}$ on the copy of $\one$ in the middle, which is equivalent to the right action of $x_i-x_n$ (by Corollary \ref{cor: rouquier auto}).
\end{proof}

\begin{remark}
The analogous statement to Theorem \ref{lem: cone alpha} holds if the map \eqref{eqn:cone alpha} is $\star$-multiplied with other $E$'s on both the left and the right. Indeed, $Y_i$ and $Y_{i+1}$ commute with all other $Y_j$ and all $T_k, k\neq i-1,i,i+1$, so the same proof works.
\end{remark}

\begin{remark}

Motivated by the geometry of $\Comm_n$ and \cite[Theorem 5.25]{GW}, we expect that the map \eqref{eqn:cone alpha} vanishes in $\Tr(\ASBim_n)$. As a consequence, $E_{(d_1,\dots,d_i,d_{i+1},\dots, d_n)}$ and  $E_{(d_1,\dots,d_{i}-1,d_{i+1}+1,\dots,d_n)}$ would be shown to be direct summands of: 
$$
\left[qE_{(d_1,\dots,d_i)}\star E_{(d_{i+1},\dots,d_n)}\xrightarrow{x_i-x_{n}} q^{-1}E_{(d_1,\dots,d_i)}\star E_{(d_{i+1},\dots,d_n)}\right].
$$
\end{remark}

\subsection{Exact triangles} We introduce the following notations to keep track of more complicated exact triangles. First, we define $\alpha_i=\ee_i-\ee_{i+1}$. Given a vector $\dd = (d_1,\ldots,d_n)$, we define $\dd[i,j]=(d_i,\ldots,d_j)$. Finally, for $\aaa=(a_1,\ldots,a_n)$ and $\bb=(b_1,\ldots,b_m)$ we let: 
$$
E_{\aaa}\bstar E_{\bb}=\left[qE_{\aaa}\star E_{\bb}\xrightarrow{x_n-x_{n+m}} q^{-1}E_{\aaa}\star E_{\bb}\right]
$$

\begin{lemma}
\label{lem: bstar}
One can extend $\bstar$ to an associative operation on objects $E_{\aaa}$.
\end{lemma}

\begin{proof}
Suppose that $\aaa=(a_1,\ldots,a_n),\bb=(b_1,\ldots,b_m)$ and $\cc=(c_1,\ldots,c_k)$. Then $(E_{\aaa}\bstar E_{\bb})\bstar E_{\cc}$ is the four term Koszul complex built of $E_{\aaa}\star E_{\bb}\star E_{\cc}$  with differentials given by $(x_{n}-x_{n+m})$ and $(x_{n+m}-x_{n+m+k})$. The other composition $E_{\aaa}\bstar (E_{\bb}\bstar E_{\cc})$ is a similar complex with differentials given by $(x_{n+m}-x_{n+m+k})$ and $(x_n-x_{n+m+k})$. Since
$$
(x_{n}-x_{n+m})+(x_{n+m}-x_{n+m+k})=(x_{n}-x_{n+m+k}),
$$
the two complexes are isomorphic via a simple change of variables.
\end{proof}

With these notations in hand, we can compactly write Theorem \ref{lem: cone alpha} as the exact triangle:
$$
\begin{tikzcd}
E_{\dd}\arrow{rr}  & & E_{\dd-\alpha_i} \arrow{dl}\\
 & E_{\dd[1,i]}\bstar E_{\dd[i+1,n]} \arrow{ul} &  
\end{tikzcd}
$$
where $\dd[i,j] = (d_i,\dots,d_j)$ for any $\dd = (d_1,\dots,d_n)$ and all $i\leq j$.

\begin{lemma}
(a) The following diagram commutes up to homotopy:
$$
\begin{tikzcd}
E_{\dd} \arrow{r}\arrow{d}& E_{\dd-\alpha_i}\arrow{d}\\
E_{\dd-\alpha_j}\arrow{r} & E_{\dd-\alpha_i-\alpha_j}
\end{tikzcd}
$$
hence there is a well-defined (up to a homotopy) map $E_{\dd}\to E_{\dd-\alpha_i-\alpha_j}$.

(b) Suppose that $i<j$. The diagram in (a) fits as the central rhombus in the following diagram:
$$
\begin{tikzcd}
E_{\dd[1,i]}\bstar E_{\dd[i+1,n]} \arrow{r} \arrow[bend right]{ddrr}& E_{\dd} \arrow{dl} \arrow{dr} & E_{\dd[1,j]}\bstar E_{\dd[j+1,n]} \arrow{l} \arrow[bend left]{ddll}\\
E_{\dd-\alpha_i} \arrow{u} \arrow{dr} & E_{\dd[1,i]}\bstar E_{\dd[i+1,j]}\bstar E_{\dd[j+1,n]} \arrow{ul} \arrow{ur}& E_{\dd-\alpha_j} \arrow{u} \arrow{dl} \\
E_{\dd[1,j]-\alpha_i}\bstar E_{\dd[j+1,n]} \arrow{u} \arrow{ur}& E_{\dd-\alpha_i-\alpha_j} \arrow{r} \arrow{l}& E_{\dd[1,i]}\bstar E_{\dd[i+1,n]-\alpha_j} \arrow{u} \arrow{ul}
\end{tikzcd}
$$
in which all oriented triangles are exact, and all oriented quadrilaterals are commutative.
\end{lemma}

\begin{proof} Part (a) is a special case of (b).
Assuming without loss of generality $i<j$, we can write: 
$$
E_{\dd}\simeq \Tr(Y^{\dd-\ee_i-\ee_j}T_{1}\cdots (T_iY_i)\cdots (T_jY_j)\cdots T_{n-1})
$$
Clearly, the maps from skein exact sequences from Lemma \ref{lem: skein 1} applied at positions $i$ and $j$ commute with each other, so we just need to check that various isomorphisms used in the proof of Theorem \ref{lem: cone alpha} commute with them as well. Indeed, we can write all the entries in the diagram in (b) as follows:

1) $E_{\dd[1,i]}\bstar E_{\dd[i+1,n]}$ is built of two copies of 
$
\Tr(Y^{\dd-\ee_i-\ee_j}T_{1}\cdots (Y_i)\cdots (T_jY_j)\cdots T_{n-1})
$;

2) $E_{\dd[1,j]}\bstar E_{\dd[j+1,n]}$ is built of two copies of 
$
\Tr(Y^{\dd-\ee_i-\ee_j}T_{1}\cdots (T_iY_i)\cdots (Y_j)\cdots T_{n-1})
$;

3) $E_{\dd-\alpha_i}=
\Tr(Y^{\dd-\ee_i-\ee_j}T_{1}\cdots (Y_{i+1}T_i)\cdots (T_jY_j)\cdots T_{n-1})
$;

4) $E_{\dd[1,i]}\bstar E_{\dd[i+1,j]}\bstar E_{\dd[j+1,n]}$ is built of four copies (see Lemma \ref{lem: bstar}) of: 
$$
\Tr(Y^{\dd-\ee_i-\ee_j}T_{1}\cdots (Y_{i})\cdots (Y_j)\cdots T_{n-1});
$$

5) $ E_{\dd-\alpha_j}=\Tr(Y^{\dd-\ee_i-\ee_j}T_{1}\cdots (T_iY_i)\cdots (Y_{j+1}T_j)\cdots T_{n-1})
$;

6) $E_{\dd[1,j]-\alpha_i}\bstar E_{\dd[j+1,n]}$ is built of two copies of 
$
\Tr(Y^{\dd-\ee_i-\ee_j}T_{1}\cdots (Y_{i+1}T_i)\cdots (Y_{j})\cdots T_{n-1})
$;

7) $E_{\dd-\alpha_i-\alpha_j}=\Tr(Y^{\dd-\ee_i-\ee_j}T_{1}\cdots (Y_{i+1}T_{i})\cdots (Y_{j+1}T_j)\cdots T_{n-1})$;

8) $E_{\dd[1,i]}\bstar E_{\dd[i+1,n]-\alpha_j}$ is built of two copies of 
$
\Tr(Y^{\dd-\ee_i-\ee_j}T_{1}\cdots (Y_{i})\cdots (Y_{j+1}T_j)\cdots T_{n-1})
$.

It remains to notice that $Y_i$ commutes with $T_j$ and $Y_j$, and $Y_{j+1}$ commutes with $T_i$ and $Y_i$. Furthermore, by Remark \ref{rem: skein 1 commute}, $Y_i$ commutes with the skein maps involving $T_j$ and $Y_j$, and $Y_{j+1}$ commutes with the skein maps involving $T_i$ and $Y_i$.  
\end{proof}

\begin{corollary}
\label{cor: skein 1}
If $\dd-\dd'$ is a non-negative linear combination of $\alpha_i$'s, then there is a well defined (up to a homotopy) map $E_{\dd}\to E_{\dd'}$, and its cone is filtered by $\overline{\star}$-products of smaller $E_{\aaa}$.
\end{corollary}

\subsection{Twisted lattices and more exact sequences}

Recall that the affine symmetric group has $n!$ different lattices in it which are conjugate to the standard lattice generated by $y_i$ under the action of $S_n$. We can mimic this construction in the affine Hecke category.

\begin{lemma}
\label{lem: y prime}
Define $Y'_i=T_i^{-1}Y_iT_i,\ Y'_{i+1}=T_i^{-1}Y_{i+1}T_i$. Then the following statements hold:
\begin{itemize}
\item [(a)] We have $Y_iY_{i+1}\simeq Y'_iY'_{i+1}$.

\item [(b)] There is a canonical map $Y_i\to Y'_{i+1}$ with cone:
$$
[Y_i\to Y'_{i+1}]\simeq [qY_{i+1}T_{i}\xrightarrow{x_i-x_{i+1}} q^{-1}Y_{i+1}T_{i}]
$$
\item [(c)] There is a canonical map $Y'_{i}\to Y_{i+1}$ with cone:
$$
[Y'_i\to Y_{i+1}]\simeq  [qY_{i+1}T_{i}\xrightarrow{x_{i}-x_{i+1}-\delta} q^{-1}Y_{i+1}T_{i}]
$$
\item [(d)] The objects $Y'_i,Y'_{i+1}$ commute with $Y_j$ for $j\neq i,i+1$.
\end{itemize}

\end{lemma}

\begin{proof}
(a) The product $Y_1\cdots Y_n=\Omega^n$ is central, so it commutes with $T_i$. On the other hand, $Y_j$ commutes with $T_i$ for $j\neq i,i+1$, so $Y_iY_{i+1}$ commutes with $T_i$ and thus:
$$
Y'_iY'_{i+1}=T_i^{-1}Y_iY_{i+1}T_{i}=Y_{i}Y_{i+1}.
$$ 

(b) We have $Y_i=T_{i}Y_{i+1}T_{i}$, which by \eqref{eq: skein 0} has a canonical map to $Y'_{i+1}=T_{i}^{-1}Y_{i+1}T_{i}$ with cone homotopy equivalent to $[qY_{i+1}T_{i} \xrightarrow{x_i-x_{i+1}} q^{-1}Y_{i+1}T_{i}]$ (note that we multiply $\one$ by $Y_{i+1}T_i$ on the right, so the variables do not change). 

(c) We have $Y'_{i}=Y_{i+1}T_{i}^2$, which by \eqref{eq: skein} has a canonical map to $Y_{i+1}$ with cone   homotopy equivalent to $[qY_{i+1}T_{i} \xrightarrow{x_{i}-x_{i+1}-\delta} q^{-1}Y_{i+1}T_i]$
(note that we multiply $T_i$ by $Y_{i+1}$ on the left, so $x_{i+1}$ is shifted by $\delta$). Part (d) is clear.
\end{proof}

Under the functor predicted by Problem \ref{prob:main}, the morphisms described in Lemma \ref{lem: y prime} should correspond to the morphisms induced by \cite[Proposition 2.28]{Hecke}. 

\begin{lemma}
\label{lem: y swap}
a) If $a\ge b$, then there is a canonical map $Y^{a}_{i}Y^{b}_{i+1} \to Y'^b_{i}Y'^{a}_{i+1}$, whose cone is filtered by objects of the form:
$$
[qY^{a-k}_{i}Y^{b+k}_{i+1}T_i\xrightarrow{x_i-x_{i+1}+(a-b-k)\delta} q^{-1}Y^{a-k}_{i}Y^{b+k}_{i+1}T_i], \quad k=1,\ldots, a-b.
$$
b) If $a\le b$, then there is a canonical map $Y'^{b}_{i}Y'^{a}_{i+1} \to Y^a_{i}Y^{b}_{i+1}$, whose cone is filtered by objects of the form:
$$
[q Y_i^{b-k} Y_{i+1}^{a+k} T_i \xrightarrow{x_i-x_{i+1}-k\delta} q^{-1}Y_i^{b-k} Y_{i+1}^{a+k} T_i], \quad k = 1, \dots, b-a
$$
\end{lemma}

\begin{proof}
a) We will first deal with the case $a\ge 0$ and $b=0$. By Lemma \ref{lem: y prime}(b), we have a chain of maps:
$$
Y^{a}_i\to Y^{a-1}_iY'_{i+1}\to Y^{a-2}_{i}Y'^2_{i+1}\to \dots \to Y'^{a}_{i+1}.
$$

The cone of their composition is filtered by the cones of the individual maps:
$$
[Y^{a-k+1}_{i}Y'^{k-1}_{i+1}\to Y^{a-k}_{i}Y'^{k}_{i+1}]=Y^{a-k}_{i}[Y_{i}\to Y'_{i+1}]Y'^{k-1}_{i+1} \simeq
$$
$$
\simeq Y^{a-k}_{i}[qY_{i+1}T_i\xrightarrow{x_i-x_{i+1}} q^{-1}Y_{i+1}T_i]Y'^{k-1}_{i+1}=[qY^{a-k}_{i}Y^{k}_{i+1}T_i\xrightarrow{x_i-x_{i+1}+(a-k)\delta} q^{-1}Y^{a-k}_{i}Y^{k}_{i+1}T_i]
$$
for all $k \in \{1,\dots,a\}$. In the last equality above, we used the equalities:
\begin{equation}
\label{eqn:a certain isomorphism}
Y_{i+1}T_iY'^{k-1}_{i+1} = Y_{i+1}Y^{k-1}_{i+1}T_i=Y^{k}_{i+1}T_i.
\end{equation}
In the case of general $a\geq b$, the discussion above for the numbers $a-b$ and $0$ implies that there is a map $Y_i^{a-b} \rightarrow Y'^{a-b}_{i+1}$ whose cone is filtered by objects of the form
$$
[qY^{a-b-k}_{i}Y^{k}_{i+1}T_i\xrightarrow{x_i-x_{i+1}+(a-b-k)\delta} q^{-1}Y^{a-b-k}_{i}Y^{k}_{i+1}T_i],\ k=1,\ldots, a-b.
$$
If we tensor the objects above on the left with $(Y_iY_{i+1})^{b}=(Y'_iY'_{i+1})^{b}$ (the equality holds due to Lemma \ref{lem: y prime}(a)), then we conclude that there is a map $Y_i^a Y_{i+1}^b \to Y'^b_i Y'^a_{i+1}$ whose cone is filtered by maps between the objects:
$$
(Y_iY_{i+1})^{b}Y^{a-b-k}_{i}Y^{k}_{i+1}T_i=Y^{a-k}_iY^{b+k}_{i+1}T_i.
$$
as $k$ goes from 1 to $a-b$.

b) Let us first deal with the case $a = 0, b \geq 0$. By Lemma \ref{lem: y prime}(c), we have a chain of maps:
$$
Y'^{b}_i\to Y_{i+1}Y'^{b-1}_i\to Y^2_{i+1}Y'^{b-2}_{i}\to \ldots \to Y^{b}_{i+1}
$$
The cone of their composition is filtered by the cones of the individual maps:
$$
[Y^{k-1}_{i+1}Y'^{b-k+1}_{i}\to Y^{k}_{i+1}Y'^{b-k}_{i}]=Y^{k-1}_{i+1}[Y'_{i}\to Y_{i+1}]Y'^{b-k}_{i} \simeq
$$
$$
\simeq Y^{k-1}_{i+1}[qY_{i+1}T_i\xrightarrow{x_i-x_{i+1}-\delta} q^{-1}Y_{i+1}T_i]Y'^{b-k}_{i}=[qY^{k}_{i+1}Y^{b-k}_{i}T_i \xrightarrow{x_i-x_{i+1}-k\delta} q^{-1}Y^{b-k}_{i}Y^{k}_{i+1}T_i]
$$
for all $k \in \{1,\dots, b\}$ (we used \eqref{eqn:a certain isomorphism} in the last equality above). The remaining argument in part (b) is similar to that of part (a).

\end{proof}

Lemma \ref{lem: y swap} is the fundamental instance of the following more general result.

\begin{lemma}
\label{lem: y swap long}
Consider an arbitrary sequence of integers $\dd=(d_1,\ldots,d_n) \in \Z^n$. \newline 

a) If $d_i\ge d_{i+1}$, then there is a chain map: 
\begin{equation}
\label{eqn:map of y's 1}
Y_1^{d_1}\cdots Y_i^{d_i}Y_{i+1}^{d_{i+1}}\cdots Y_n^{d_n}\to T_i^{-1}Y_1^{d_1}\cdots Y_i^{d_{i+1}}Y_{i+1}^{d_{i}}\cdots Y_n^{d_n}T_i\
\end{equation}
whose cone is filtered by the two-term complexes:
\begin{equation}
\label{eqn:cone map of y's 1}
\left[qY_1^{d_1}\cdots Y_i^{d_i-k}Y_{i+1}^{d_{i+1}+k}\cdots Y_n^{d_n}T_i\xrightarrow{x_i-x_{i+1}+(d_i-d_{i+1}-k)\delta} q^{-1}Y_1^{d_1}\cdots Y_i^{d_i-k}Y_{i+1}^{d_{i+1}+k}\cdots Y_n^{d_n}T_i\right]
\end{equation}
for $k=1,\ldots, d_i-d_{i+1}$. \newline

b) If $d_i\le d_{i+1}$, then there is a chain map: 
\begin{equation}
\label{eqn:map of y's 2}
T_i^{-1}Y_1^{d_1}\cdots Y_i^{d_{i+1}}Y_{i+1}^{d_{i}}\cdots Y_n^{d_n}T_i \to Y_1^{d_1}\cdots Y_i^{d_i}Y_{i+1}^{d_{i+1}}\cdots Y_n^{d_n}
\end{equation}
whose cone is filtered by the two-term complexes:
\begin{equation}
\label{eqn:cone map of y's 2}
\left[qY_1^{d_1}\cdots Y_i^{d_{i+1}-k}Y_{i+1}^{d_{i}+k}\cdots Y_n^{d_n}T_i\xrightarrow{x_i-x_{i+1}-k\delta} q^{-1}Y_1^{d_1}\cdots Y_i^{d_{i+1}-k}Y_{i+1}^{d_{i}+k}\cdots Y_n^{d_n}T_i\right]
\end{equation}
for $k=1,\ldots, d_{i+1}-d_{i}$.

\end{lemma}

\begin{proof} Let us prove case (a), and leave case (b) as an analogous exercise to the interested reader. By Lemma \ref{lem: y swap} in case (a) we get a map:
$$
Y_1^{d_1}\cdots Y_i^{d_i}Y_{i+1}^{d_{i+1}}\cdots Y_n^{d_n}\to Y_1^{d_1}\cdots Y'^{d_{i+1}}_iY'^{d_{i}}_{i+1}\cdots Y_n^{d_n}
$$
whose cone is filtered by the two-term complexes:
$$
\left[qY_1^{d_1}\cdots Y_i^{d_i-k}Y_{i+1}^{d_{i+1}+k}T_i\cdots Y_n^{d_n}\xrightarrow{x_i-x_{i+1}+(d_i-d_{i+1}-k)\delta} q^{-1}Y_1^{d_1}\cdots Y_i^{d_i-k}Y_{i+1}^{d_{i+1}+k}T_{i}\cdots Y_n^{d_n}\right]
$$
However, note that:
\begin{equation}
\label{eqn:remark}
Y'^a_{i}Y'^{b}_{i+1}=T_i^{-1}Y^a_iY^b_{i+1}T_i
\end{equation}
for all integers $a,b$. Together with the fact that $T_i$ commutes with $Y_j$ for $j\neq i,i+1$, we conclude that there exists a map \eqref{eqn:map of y's 1} whose cone is filtered by the complexes \eqref{eqn:cone map of y's 1}. 

\end{proof}

\subsection{Commutators in the Trace} We will now use Lemma \ref{lem: y swap long} to prove Theorem \ref{thm:rel 2}.

\begin{theorem}
\label{thm: one step commutation}

For any $\dd = (d_1,\dots,d_n) \in \Z^n$ and $k \in \Z$, there exists a collection of objects $G_0,\dots,G_n \in \ASBim_n$ with the following properties:

\begin{itemize}
    
    \item $\emph{Tr}(G_0)=E_{(k)} \star E_{\dd}$ and $\emph{Tr}(G_n)=E_{\dd} \star E_{(k)}$ 
    
    \item For all $i\in \{1,\ldots,n\}$ there exist chain maps in $\CK(\ASBim_n)$:
    $$
    \begin{cases}
    G_{i-1}\xrightarrow{\varphi_i} G_i & \text{if}\ k \geq d_i\\
    G_{i-1}\xleftarrow{\overline{\varphi}_i} G_i & \text{if}\ k \leq d_i
    \end{cases}
    $$
    which are mutually inverse isomorphisms if $k = d_i$.
    
    \item $\emph{Tr}(\emph{Cone}(\varphi_i))$ and $\emph{Tr}(\emph{Cone}(\overline{\varphi}_i))$ are filtered by:
    \begin{equation}
    \label{eqn:cone top 2}
    [\C_{q} \xrightarrow{0} \C_{q^{-1}}] \cdot \begin{cases}
    E_{(d_1,\ldots,d_{i-1},k-a,d_{i}+a,d_{i+1},\ldots,d_n)},\ 1\le a\le k-d_i &  \text{if}\ k>d_i\\
    E_{(d_1,\ldots,d_{i-1},d_i-a,k+a,d_{i+1},\ldots,d_n)},\ 1\le a\le d_i-k &  \text{if}\ k<d_i
    \end{cases}
    \end{equation}
    respectively. Above, $q$ denotes the internal grading on Soergel bimodules.
    
\end{itemize}
\end{theorem}

\begin{proof}

Define the following sequence of objects in $\CK(\ASBim_n)$: 
$$
G_i=T_1^{-1}\cdots T_{i}^{-1}Y_1^{d_1}\cdots Y_{i}^{d_{i}}Y_{i+1}^{k}Y_{i+2}^{d_{i+1}}\cdots Y_{n+1}^{d_n}T_{i}\cdots T_{1}T_2\cdots T_{n} 
$$
for all $i \in \{0,\dots,n\}$. We have $\Tr(G_0) = E_{(k)} \star E_{\dd}$ by definition, while the identity $T_1 T_2 \dots T_n \dots T_2 T_1 = T_n \dots T_2 T_1 T_2 \dots T_n$ in the braid group implies that:
\begin{align*}
\Tr(G_n) &= \Tr(T_1^{-1}\cdots T_{n}^{-1}Y_{1}^{d_{1}}\cdots Y_{n}^{d_n}Y_{i+1}^{k} T_{n} \dots T_2 T_1 T_2 \dots T_n) = \\ &= \Tr(T_1^{-1}\cdots T_{n}^{-1}Y_{1}^{d_{1}}\cdots Y_{n}^{d_n}Y_{i+1}^{k} T_{1} T_2  \dots T_n \dots T_2 T_1) = \\
&= \Tr(Y_{1}^{d_{1}}\cdots Y_{n}^{d_n}Y_{i+1}^{k} T_{1} T_2  \dots T_n) = E_{\dd} \star E_{(k)}
\end{align*}
This establishes (a). As for (b), we note that Lemma \ref{lem: y swap long} gives us chain maps  $\varphi_i : G_{i-1} \to G_{i}$ if $d_{i}\le k$, and $\overline{\varphi}_i : G_{i}\to G_{i-1}$ if $d_{i}\ge k$, whose cones are filtered by two copies of: 
\begin{equation}
\label{eqn:big 1}
\left\{ T_1^{-1}\cdots T_{i-1}^{-1}Y_1^{d_1}\cdots Y_{i-1}^{d_{i-1}}Y_{i}^{k - a}Y_{i+1}^{d_{i} + a} Y_{i+2}^{d_{i+1}} \cdots Y_{n+1}^{d_n}T_{i} T_{i-1}\cdots T_{1}T_2\cdots T_{n}\right\}_{1 \leq a \leq k-d_i}
\end{equation}
if $d_i \leq k$, and:
\begin{equation}
\label{eqn:big 2}
\left\{ T_1^{-1}\cdots T_{i-1}^{-1}Y_1^{d_1}\cdots Y_{i-1}^{d_{i-1}}Y_{i}^{d_i-a}Y_{i+1}^{k+a} Y_{i+2}^{d_{i+1}} \cdots Y_{n+1}^{d_n}T_{i}T_{i-1}\cdots T_{1}T_2\cdots T_{n} \right\}_{1 \leq a \leq d_i-k}
\end{equation}
if $d_i \geq k$. Since $T_{i}T_{i-1}\cdots T_{1}T_2\cdots T_{n}=T_{1}T_2\cdots T_{n}T_{i-1}\cdots T_1$ holds in the braid group, then:
$$
\Tr\left[T_1^{-1}\cdots T_{i-1}^{-1}Y_1^{s_1} \cdots Y_{n+1}^{s_{n+1}}T_{i}T_{i-1}\cdots T_{1}T_2\cdots T_{n}\right] = \Tr\left[Y_1^{s_1} \cdots Y_{n+1}^{s_{n+1}}T_{1}T_2\cdots T_{n}\right] = E_{(s_1,\dots,s_{n+1})}
$$
for all integers $s_1,\dots,s_{n+1}$. Thus, the traces of the objects \eqref{eqn:big 1} and \eqref{eqn:big 2} are precisely the $E$'s that appear in \eqref{eqn:cone top 2}. As for the differential maps between the two copies of said $E$'s, by Lemma \ref{lem: y swap long} they are given by $x_i - x_{i+1} + m\delta$ for various integers $m$. As Lemma \ref{lem: endo trace rouquier} ensures that $x_i=x_{i+1}$ and $\delta=0$ on $E_{\dd}$, these differential maps vanish.
\end{proof}

\begin{remark}
Note that the projections of $E_{(k)}\star E_{\dd}$ and $E_{\dd}\star E_{(k)}$ to $\widetilde{S_n}$ are, respectively:
$$
y_1^{k}y_2^{d_1}\cdots y_{n}^{d_{n-1}}s_2\cdots s_{n-1} \quad \text{and} \quad y_1^{d_1}\cdots y_{n-1}^{d_{n-1}}y_{n}^{k}s_1\cdots s_{n-2}
$$ 
which are conjugate in the affine symmetric group. The conjugating element is $s_1\cdots s_{n-1}$, which explains the structure of the proof of Theorem \ref{thm: one step commutation}: indeed, these elements are no longer conjugate in the affine Hecke category, but at each step we conjugate by one simple reflection and apply Lemma  \ref{lem: y swap long} to collect correction terms. 
\end{remark}

\begin{remark}
\label{rem: one step commutator with the product}
A very similar proof applies for the commutation relation between $E_{(k)}$ and a $\star$ product of $E_{\dd}$. The definition of $G_i$ now reads:
$$
G_i=T_1^{-1}\cdots T_{i}^{-1}Y_1^{d_1}\cdots Y_{i}^{d_{i}}Y_{i+1}^{k}Y_{i+2}^{d_{i+1}}\cdots Y_{n+1}^{d_n}T_{i}\cdots T_{1}T_{u}
$$
where $u$ is a subword of $s_2\cdots s_n$. The chain maps and their cones are computed as in the proof of Theorem \ref{thm: one step commutation}, and the only nontrivial computation is:
$$
T_i\cdots T_1 T_{u} T_1^{-1}\cdots T_{i-1}^{-1}=T_i\cdots T_1 T_{v} T_{w} T_1^{-1}\cdots T_{i-1}^{-1}=
T_{\overline{v}}T_i\cdots T_1  T_1^{-1}\cdots T_{i-1}^{-1} T_{w}=T_{\overline{v}} T_i T_{w},
$$
where $u=vw$ with $v$ being a product of $s_j, 2\le j\le i$, $w$ being a product of $s_j, j\ge i+1$, and $\overline{v}$ being obtained from $v$ by decreasing all indices by 1. We leave the details as an exercise to the reader, and provide a particular instance of this computation below.

\end{remark}

\begin{example}
For example, for $[E_{(d_1,d_2)}\star E_{(d_3,d_4)},E_{(k)}]$, we have $T_u=T_2T_4$ and thus:
\begin{align*}
\Tr(G_0) &= \Tr(Y_1^{k}Y_2^{d_1}Y_3^{d_2}Y_4^{d_3}Y_5^{d_4}T_2T_4)=E_{(k)}\star E_{(d_1,d_2)}\star E_{(d_3,d_4)}  \\
\Tr(G_1) &=\Tr(T_1^{-1}Y_1^{d_1}Y_2^{k}Y_3^{d_2}Y_4^{d_3}Y_5^{d_4}T_1T_2T_4) \\
\Tr(G_2) &=\Tr(T_1^{-1}T_2^{-1}Y_1^{d_1}Y_2^{d_2}Y_3^{k}Y_4^{d_3}Y_5^{d_4}T_2T_1T_2T_4)=\Tr(Y_1^{d_1}Y_2^{d_2}Y_3^{k}Y_4^{d_3}Y_5^{d_4}T_2T_1T_2T_4T_1^{-1}T_2^{-1})\\
&=\Tr(Y_1^{d_1}Y_2^{d_2}Y_3^{k}Y_4^{d_3}Y_5^{d_4}T_1T_4)=E_{(d_1,d_2)}\star E_{(k)}\star E_{(d_3,d_4)} \\
\Tr(G_3) &=\Tr(T_1^{-1}T_2^{-1}T_3^{-1}Y_1^{d_1}Y_2^{d_2}Y_3^{d_3}Y_4^{k}Y_5^{d_4}T_3T_2T_1T_2T_4)\\
&=\Tr(Y_1^{d_1}Y_2^{d_2}Y_3^{d_3}Y_4^{k}Y_5^{d_4}T_3T_2T_1T_2T_4T_1^{-1}T_2^{-1}T_3^{-1})=
\Tr(Y_1^{d_1}Y_2^{d_2}Y_3^{d_3}Y_4^{k}Y_5^{d_4}T_1T_3T_4T_3^{-1})\\
\Tr(G_4)&=\Tr(T_1^{-1}T_2^{-1}T_3^{-1}T_4^{-1}Y_1^{d_1}Y_2^{d_2}Y_3^{d_3}Y_4^{d_4}Y_5^{k}T_4T_3T_2T_1T_2T_4)\\
&=\Tr(Y_1^{d_1}Y_2^{d_2}Y_3^{d_3}Y_4^{d_4}Y_5^{k}T_4T_3T_2T_1T_2T_4T_1^{-1}T_2^{-1}T_3^{-1}T_4^{-1})\\
&=\Tr(Y_1^{d_1}Y_2^{d_2}Y_3^{d_3}Y_4^{d_4}Y_5^{k}T_1T_3)=E_{(d_1,d_2)}\star E_{(d_3,d_4)}\star E_{(k)}.
\end{align*}

\end{example}

\subsection{Application: skein product on cocenters of $\AH_n$}

We can summarize the preceding results as follows. By \eqref{eq: K0 tr} we have a surjective map:
$$
\Tr(G(\ASBim_n))=\Tr(\AH_n)\to G(\Tr(\ASBim_n)).
$$

\begin{theorem}
\label{thm:surj}

Let $\widetilde{\CA}$ be the algebra from Definition \ref{def:a}. There is a surjective algebra homomorphism: 
$$
\widetilde{\CA}\Big|_{(q_1,q_2)\rightarrow (q^{-2},q^{2})}\to \bigoplus_{n=0}^{\infty} \Tr(\AH_n)
$$
where the multiplication in the right hand side is given by the skein product $\AH_n\times \AH_k\to \AH_{n+k}$. 
\end{theorem}

\begin{proof}
The generators $E_{\dd^1}\star \dots\star E_{\dd^r}$ from Lemma \ref{lem: trace generate}, being the traces of objects in $\CK(\ASBim_n)$,  correspond to some explicit elements of $\Tr(\AH_n)$.
The proof of Lemma \ref{lem: trace generate} and earlier Proposition \ref{prop:generate} use \underline{only} the skein exact sequences from $\ASBim_n$ and the relation $\Tr(XY)\simeq \Tr(YX)$, so they correspond to relations in $\Tr(\AH_n)$. In particular, the classes $E_{\dd^1}\star \dots\star E_{\dd^r}$ generate $\Tr(\AH_n)$ over $\Z[q^{\pm 1}]$. Similarly, the exact sequences from Theorem \ref{lem: cone alpha} and  Theorem \ref{thm: one step commutation} lead to relations in $\Tr(\AH_n)$ which match the relations \eqref{eqn:rel a 1} and \eqref{eqn:rel a 2} in the algebra $\widetilde{\CA}$ after specialization $(q_1,q_2) \mapsto (q^{-2},q^{2})$, upon the substitution: 
$$
E_{(d_1,\dots,d_n)} = q^{n-1} \EE_{(d_1,\dots,d_n)}.
$$
\end{proof}

\begin{corollary}
There is a surjective map $\widetilde{\CA}\Big|_{(q_1,q_2)\rightarrow (q^{-2},q^{2})} \to \bigoplus_{n=0}^{\infty} G(\Tr(\ASBim_n))$.
\end{corollary}

\end{document}